\documentclass[a4paper]{amsart}
\pdfoutput=1
\usepackage[pdftex,
            pdfauthor={Alexis Marchand},
            pdftitle={From letter-quasimorphisms to angle structures and spectral gaps for scl},
            colorlinks=true]{hyperref}
\usepackage[T1]{fontenc}
\usepackage[english]{babel}
\usepackage{lmodern}
\usepackage[dvipsnames]{xcolor}
\usepackage{graphicx}
\usepackage{enumitem}
\usepackage{amsmath,amsthm,amssymb,mathrsfs,eqnarray,relsize,bm,mathtools}
\usepackage[msc-links]{amsrefs}
\usepackage{caption,subcaption}

\usepackage{todonotes,comment}

\usepackage{tikz}
\usetikzlibrary{tqft}
\usetikzlibrary{calc}
\usetikzlibrary{shapes,arrows,decorations.markings}
\usetikzlibrary[patterns,patterns.meta]
   
\theoremstyle{plain}
\newtheorem{prop}{Proposition}
\numberwithin{prop}{section}
\newtheorem{thm}[prop]{Theorem}
\newtheorem{cor}[prop]{Corollary}
\newtheorem{lmm}[prop]{Lemma}

\newtheorem{innermthm}{Theorem}
\newenvironment{mthm}[1]
  {\renewcommand\theinnermthm{#1}\innermthm}
  {\endinnermthm}

\theoremstyle{definition}
\newtheorem{rmk}[prop]{Remark}
\newtheorem{ex}[prop]{Example}
\newtheorem{defi}[prop]{Definition}

\newtheorem{claim}{Claim}
\newtheorem*{claim*}{Claim}

\newcommand{\st}{\ensuremath{\:\middle\vert\:}}

\renewcommand{\emptyset}{\varnothing}

\DeclareMathOperator{\scl}{scl}

\DeclareMathOperator{\lk}{Lk}
\DeclareMathOperator{\CAT}{CAT}

\DeclareMathOperator{\sign}{sign}
\DeclareMathOperator{\tot}{tot}
\DeclareMathOperator{\intr}{int}

\title[From letter-quasimorphisms to angle structures]{From letter-quasimorphisms to angle structures and spectral gaps for scl}
\author{Alexis Marchand}
\date{\today}
\address{DPMMS, Centre for Mathematical Sciences, Wilberforce Road, Cambridge CB3 0WB, United Kingdom}
\email{\href{mailto:aptm3@cam.ac.uk}{aptm3@cam.ac.uk}}

\begin{document}

\begin{abstract}
    We give a new geometric proof of a theorem of Heuer showing that, in the presence of letter-quasimorphisms (which are analogues of real-valued quasimorphisms with image in free groups), and in particular in RAAGs, there is a sharp lower bound of $\frac{1}{2}$ for stable commutator length.
    Our approach is to show that letter-quasimorphisms give rise to negatively curved angle structures on admissible surfaces.
    This generalises Duncan and Howie's proof of the $\frac{1}{2}$-lower bound in free groups, and can also be seen as a version of Bavard duality for letter-quasimorphisms.
\end{abstract}

\maketitle
    
\section{Introduction}

    Stable commutator length, or $\scl$, is an invariant of group elements that measures their homological complexity.
    Topologically, given a group $G$ realised as the fundamental group of a space $X$, and given a homologically trivial element $g\in G$, the \emph{stable commutator length} of $g$ is defined as
    \[
        \scl(g)\coloneqq\inf_{f,\Sigma}\frac{-\chi^-(\Sigma)}{2n(\Sigma)},
    \]
    where the infimum is over all \emph{admissible surfaces}, i.e. maps $f:\left(\Sigma,\partial\Sigma\right)\rightarrow\left(X,g\right)$ from oriented compact surfaces, bounding $g^{n(\Sigma)}$ for some $n(\Sigma)\in\mathbb{N}$ --- see \S{}\ref{sec:scl} for a precise definition.
    
    Dually, Bavard \cite{bavard} showed that $\scl$ can be interpreted in terms of quasimorphisms:
    \[
        \scl(g)=\sup_{\phi}\frac{\phi(g)}{2D(\phi)},
    \]
    where the supremum is over all maps $\phi:G\rightarrow\mathbb{R}$ which satisfy
    \[
        D(\phi)\coloneqq\sup_{a,b\in G}\left|\phi(ab)-\phi(a)-\phi(b)\right|<\infty,
    \]
    and $\phi\left(a^n\right)=n\phi(a)$ for all $a\in G$ and $n\in\mathbb{Z}$ --- such maps are called \emph{homogeneous quasimorphisms}.
    See Calegari's book \cite{cal-scl}*{\S{}2.5} for more on Bavard duality.
    
    This paper is about proving \emph{lower bounds} for $\scl$.
    A guiding thread is that $\scl$ takes large values in the presence of negative curvature.
    So-called spectral gaps (defined in \S{}\ref{subsec:specgaps}) have been discovered in groups with various notions of hyperbolicity or non-positive curvature, including Gromov-hyperbolic groups \cite{calegari-fujiwara}, mapping class groups \cite{bbf}, $3$-manifold groups \cite{chen-heuer}, and certain amalgamated free products \cites{cfl,chen-specgaps}.
    Of particular importance is the following:
    \begin{thm}[Duncan--Howie \cite{duncan-howie}]
        Suppose that there is a morphism $\varphi:G\rightarrow F$, with $F$ a free group, such that $\varphi(g)\neq1$.
        Then $\scl(g)\geq\frac{1}{2}$.
        
        In particular, (residually) free groups satisfy a strong spectral gap for $\scl$.
    \end{thm}
    \emph{Strong spectral gaps} are sharp spectral gaps at $\frac{1}{2}$ with an additional assumption of non-vanishing of $\scl$ outside the trivial element (see Definition \ref{def:str-sp-gap}).
    Strong spectral gaps are inherited by subgroups, so they could reveal information about the subgroup structure of certain groups.
    
    The Duncan--Howie Theorem was generalised by Heuer \cite{heuer-raags}, who showed that right-angled Artin groups (RAAGs) also satisfy a strong spectral gap for $\scl$.
    A remarkable corollary is that the fundamental group of any special cube complex has a strong spectral gap.
    Other authors had previously studied spectral gaps in RAAGs, including Fern\'{o}s et al. \cite{fft}, and Forester et al. \cite{fst}.
    
    In order to obtain a strong spectral gap for RAAGs, Heuer \cite{heuer-raags} introduced \emph{letter-quasimorphisms}, which are analogues of quasimorphisms with image in free groups (see Definition \ref{def:let-qm}).
    We will say that a group is \emph{quasi-residually free} if every element admits a letter-quasimorphism with sufficiently non-trivial image (see Definition \ref{def:qrf}); examples include residually free groups and RAAGs.
    Heuer's main result is the following:
    \begin{thm}[Heuer \cite{heuer-raags}]\label{thm:heuer}
        Suppose that there is a letter-quasimorphism $\Phi:G\rightarrow\mathcal{A}$ with $\Phi(g)\neq1$ and $\Phi\left(g^n\right)=\Phi(g)^n$ for all $n\in\mathbb{Z}$.
        Then $\scl(g)\geq\frac{1}{2}$.
        
        In particular, quasi-residually free groups satisfy a strong spectral gap for $\scl$.
    \end{thm}
    Heuer's Theorem can be seen as a coarse version of the Duncan--Howie Theorem; it can also be thought of as a kind of Bavard duality for letter-quasimorphisms.
    
    The main goal of this paper is to obtain Theorem \ref{thm:heuer} from a geometric route that makes apparent the role played by negative curvature.
    Heuer's strategy of proof \cite{heuer-raags} is to use letter-quasimorphisms to construct real-valued quasimorphisms and apply classical Bavard duality.
    Our approach is to show that, in the presence of letter-quasimorphisms, any admissible surface admits a cellular decomposition giving rise to a certain negatively curved geometric structure.
    This geometric structure is an \emph{angle structure} in the sense of Wise \cite{wise:ang} --- see \S{}\ref{sec:ang-str} below for a definition.
    We give an upper bound on the curvature of admissible surfaces for this angle structure, which yields an estimate of their Euler characteristic via the Combinatorial Gauß--Bonnet Formula.
    This approach can be seen as a geometric reinterpretation of Chen's linear programming duality method \cites{chen,chen-specgaps,chen-heuer,chen:kervaire} --- see Remark \ref{rmk:lp-dual}.
    
    This is the content of our main theorem:
    
    \begin{mthm}{\ref{thm:exist-ang-str}}
        Let $X$ be a connected $2$-complex and let $g\in\pi_1X\smallsetminus1$.
        Assume that there is a letter-quasimorphism $\Phi:\pi_1X\rightarrow\mathcal{A}$ with $\Phi(g)\neq1$ and $\Phi\left(g^n\right)=\Phi(g)^n$ for all $n\in\mathbb{Z}$.
        
        Then given a monotone, incompressible, disc- and sphere-free admissible surface $f:(\Sigma,\partial\Sigma)\rightarrow(X,g)$ with $f_*\left[\partial\Sigma\right]=n(\Sigma)\left[S^1\right]$ for some $n(\Sigma)\in\mathbb{N}_{\geq1}$, there is an angle structure on $\Sigma$ whose total curvature satisfies
        \[
            \kappa\left(\Sigma\right)\leq-2\pi\cdot n(\Sigma).
        \]
        
        In particular, $\scl(g)\geq \frac{1}{2}$.
    \end{mthm}
    
    There are several advantages to our geometric route.
    One of them is that it emphasizes the kinship of Heuer's Theorem with the Duncan--Howie Theorem on the one hand, and the Bavard Duality Theorem on the other hand.
    Indeed, Duncan and Howie's proof \cite{duncan-howie} uses a language that is very close to Wise's angle structures, and our argument is of a similar flavour.
    Likewise, Bavard's work \cite{bavard} shows how the data of a quasimorphism relates to topological information on admissible surfaces, and our approach generalises this to letter-quasimorphisms.
    Hence, our proof brings together three of the most important theorems to date on lower bounds for $\scl$.
    The hope is that, by bringing negative curvature to the forefront, our method could lead to further spectral gap results in other hyperbolic-like settings.
    
    \subsection*{Outline of the paper} We start by recalling the topological definition of $\scl$ and a few basic facts about the reduction of admissible surfaces in \S{}\ref{sec:scl}.
    We then go on to define letter-quasimorphisms in \S{}\ref{sec:let-qm}, following Heuer \cite{heuer-raags}.
    In \S{}\ref{sec:ang-str}, we introduce angle structures, following Wise \cite{wise:ang}.
    The core of the paper is \S{}\ref{sec:lqm-ang}, in which we explain how letter-quasimorphisms give rise to negatively curved angle structures on admissible surfaces, leading to Theorem \ref{thm:exist-ang-str}.
    
    \subsection*{Acknowledgements}
    As always, my gratitude goes to Henry Wilton for his constant support and enthusiasm, and for the always helpful mathematical discussions.
    Thanks are also due to Lvzhou Chen for his comments and interest in this work, to the anonymous referee for helpful suggestions, and to Clara Löh and Oscar Randal-Williams, whose comments on the author's PhD thesis \cite{m:phd} have also led to improvements of the present paper.
    This work was financially supported by an EPSRC PhD studentship.
    
\section{Stable commutator length}\label{sec:scl}
    
    \subsection{Topological definition}
    
    We briefly recall the topological definition of $\scl$ --- for a detailed treatment, we refer the reader to Calegari's monograph \cite{cal-scl}.
    
    Let $G$ be a group, and let $X$ be a connected topological space with $\pi_1X=G$.
    Pick an element $g\in G$, whose conjugacy class corresponds to the free homotopy class of a loop $\gamma:S^1\rightarrow X$.
    Note that the circle $S^1$ is oriented: we have a well-defined fundamental class $\left[S^1\right]\in H_1\left(S^1\right)$.
    An \emph{admissible surface} for $g$ in $X$ is the data of an oriented compact surface $\Sigma$, and of continuous maps $f:\Sigma\rightarrow X$ and $\partial f:\partial\Sigma\rightarrow S^1$, such that the following diagram commutes:
    \begin{equation*}
        \vcenter{\hbox{\begin{tikzpicture}[every node/.style={draw=none,fill=none,rectangle}]
            \node (A) at (0,1.25) {$\partial\Sigma$};
            \node (B) at (2,1.25) {$\Sigma$};
            \node (Ap) at (0,0) {$S^1$};
            \node (Bp) at (2,0) {$X$};
            
            \draw [->] (A) -> (B) node [midway,above] {$\iota$};
            \draw [->] (Ap) -> (Bp) node [midway,above] {$\gamma$};
            \draw [->] (A) -> (Ap) node [midway,left] {$\partial f$};
            \draw [->] (B) -> (Bp) node [midway,left] {$f$};
        \end{tikzpicture}}}
        \label{eq:comm-diag-adm-surf}
    \end{equation*}
    where $\iota:\partial\Sigma\hookrightarrow\Sigma$ is the inclusion.
    Such an admissible surface will be denoted by $f:\left(\Sigma,\partial\Sigma\right)\rightarrow\left(X,g\right)$.
    
    The complexity of a connected surface $\Sigma$ is measured by its \emph{reduced Euler characteristic} $\chi^-\left(\Sigma\right)\coloneqq\min\left\{0,\chi(\Sigma)\right\}$.
    If $\Sigma$ is disconnected, we set $\chi^-(\Sigma)\coloneqq\sum_K\chi^-(K)$, where the sum is over all connected components of $\Sigma$.
    
    \begin{defi}
        The \emph{stable commutator length} of an element $g\in G$ is
        \[
            \scl(g)\coloneqq\inf_{f,\Sigma}\frac{-\chi^-(\Sigma)}{2n(\Sigma)},
        \]
        where the infimum is over all admissible surfaces $f:(\Sigma,\partial\Sigma)\rightarrow(X,g)$ with $\partial f_*\left[\partial\Sigma\right]=n(\Sigma)\left[S^1\right]$ in $H_1\left(S^1\right)$, for some $n(\Sigma)\in\mathbb{N}_{\geq1}$.
        We agree that $\inf\emptyset=\infty$.
    \end{defi}
    
    \subsection{Spectral gaps}\label{subsec:specgaps}
    This paper is concerned with lower bounds for $\scl$.
    More precisely, we will say that a group $G$ has a \emph{spectral gap} if there is a constant $\varepsilon>0$ such that any $g\in G$ with $\scl(g)>0$ has $\scl(g)\geq\varepsilon$.
    
    It is well-known that in any group $G$, a non-trivial commutator always has $\scl\leq\frac{1}{2}$ (because a commutator is always bounded by a torus with one boundary component).
    Hence, the best spectral gap that can occur when $\scl$ does not vanish on the derived subgroup is $\varepsilon=\frac{1}{2}$, which motivates the following definition:
    
    \begin{defi}\label{def:str-sp-gap}
        A group $G$ is said to have a \emph{strong spectral gap} if
        \[
            \forall g\in G\smallsetminus1,\:\scl(g)\geq\frac{1}{2}.
        \]
    \end{defi}
    
    It is notable that the property of having a strong spectral gap descends to subgroups by monotonicity of $\scl$.
    
    Groups that are known to have a strong spectral gap include residually free groups by work of Duncan and Howie \cite{duncan-howie}.
    Heuer \cite{heuer-raags} gave a far-ranging generalisation of the Duncan--Howie Theorem, proving that groups with enough letter-quasimorphisms (including all right-angled Artin groups and their subgroups) have a strong spectral gap.
    The goal of this paper is to give an alternative account of Heuer's Theorem, with a more topological proof that is closer in spirit to Duncan and Howie's.
    
    \subsection{Reduction of admissible surfaces}
    It is standard to assume that admissible surfaces for $\scl$ have certain topological properties:
    
    \begin{defi}
        An admissible surface $f:(\Sigma,\partial\Sigma)\rightarrow(X,g)$ is said to be
        \begin{itemize}
            \item \emph{Disc- and sphere-free} if no connected component of $\Sigma$ is a disc or a sphere,
            \item \emph{Incompressible} if every simple closed curve in $\Sigma$ with nullhomotopic image in $X$ is nullhomotopic in $\Sigma$,
            \item \emph{Monotone} if the restriction of $\partial f:\partial\Sigma\rightarrow S^1$ to every boundary component of $\Sigma$ has positive degree.
        \end{itemize}
    \end{defi}
    
    \begin{prop}[Reduction]\label{prop:monot-incompr}
        Consider an admissible surface $f:(\Sigma,\partial\Sigma)\rightarrow(X,g)$ with $\partial f_*\left[\partial\Sigma\right]=n(\Sigma)\left[S^1\right]$ for some $n(\Sigma)\in\mathbb{N}_{\geq1}$.
        
        Then there is a monotone, incompressible, disc- and sphere-free admissible surface $f':\left(\Sigma',\partial\Sigma'\right)\rightarrow(X,g)$ with $\partial f'_*\left[\partial\Sigma'\right]=n\left(\Sigma'\right)\left[S^1\right]$ for some $n\left(\Sigma'\right)\in\mathbb{N}_{\geq1}$, and such that
        \[
            \frac{-\chi^-\left(\Sigma'\right)}{2n\left(\Sigma'\right)}\leq\frac{-\chi^-\left(\Sigma\right)}{2n(\Sigma)}.
        \]
    \end{prop}
    
    \begin{proof}
        To obtain a disc- and sphere-free surface, it suffices to throw away any disc or sphere component as this does not change the reduced Euler characteristic.
        For incompressibility, the idea is that one can glue a disc onto a non-contractible simple closed curve with nullhomotopic image, thereby reducing its complexity.
        For monotonicity, see for example \cite{cal-scl}*{Proposition 2.13}.
    \end{proof}
    
    \begin{prop}[Transversality]\label{prop:transv}
        Let $X$ be a connected $2$-complex, $g\in\pi_1X$.
        Then any incompressible admissible surface $f:(\Sigma,\partial\Sigma)\rightarrow(X,g)$ may be homotoped to one that is \emph{transverse}, in the sense that it has a decomposition into the following pieces:
        \begin{itemize}
            \item \emph{Vertex discs}, i.e. discs mapping to vertices of $X$,
            \item \emph{$1$-handles}, i.e. trivial $I$-bundles over edges of $X$, and
            \item \emph{Cellular discs}, i.e. discs mapping homeomorphically onto $2$-cells of $X$.
        \end{itemize}
    \end{prop}
    
    \begin{proof}
        This follows from Buoncristiano, Rourke and Sanderson's Transversality Theorem \cite{brs}*{\S{}VII.2}.
    \end{proof}
    
    Propositions \ref{prop:monot-incompr} and \ref{prop:transv} say that, for the purpose of computing $\scl$, we can restrict our attention to admissible surfaces which are transverse, disc- and sphere-free, incompressible, and monotone.
    Figure \ref{fig:transv-srf} represents a part of such a surface.
        
        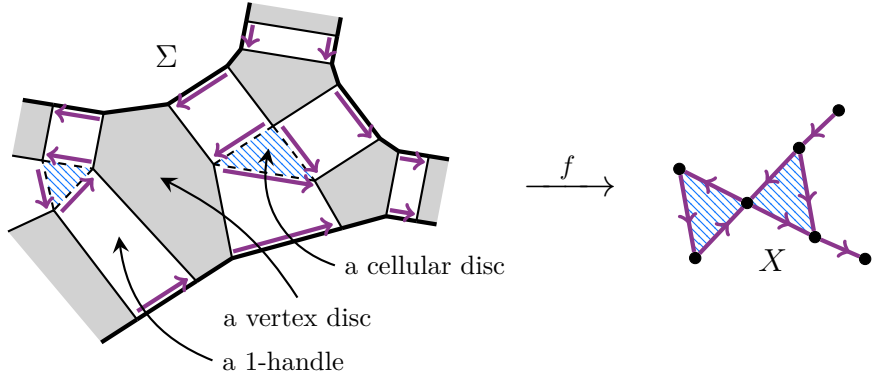
\begin{figure}[hbt]
            \centering
            \begin{subfigure}[c]{19em}
                \centering
                \begin{tikzpicture}[every node/.style={rectangle,draw=none,fill=none},xscale=.27,yscale=.25,rotate=-10]
                    \coordinate (x0) at (-5.5,-5.5);
                    \coordinate (x1) at (-4,-4);
                    \coordinate (x2) at (-1.5,-1.5);
                    \coordinate (x3) at (0,0);
                    \coordinate (x4) at (5,2.5);
                    \coordinate (x5) at (7,3.5);
                    \coordinate (x6) at (8.5,3.5);
                    \coordinate (x7) at (9.5,3.5);
                    
                    \coordinate (y0) at (9.5,7);
                    \coordinate (y1) at (8.5,7);
                    \coordinate (y2) at (7,7);
                    \coordinate (y3) at (6,7.5);
                    \coordinate (y4) at (3.5,10);
                    \coordinate (y5) at (3,11);
                    \coordinate (y6) at (3,12.5);
                    \coordinate (y7) at (3,13.5);
                    
                    \coordinate (z0) at (-1.5,13.5);
                    \coordinate (z1) at (-1.5,12.5);
                    \coordinate (z2) at (-1.5,11);
                    \coordinate (z3) at (-2,10);
                    \coordinate (z4) at (-4.5,7.5);
                    \coordinate (z5) at (-7.5,6.5);
                    \coordinate (z6) at (-10,6.5);
                    \coordinate (z7) at (-11.5,6.5);
                        
                    \coordinate (u0) at (-1.75,4.75);
                    \coordinate (u1) at (.75,7.25);
                    \coordinate (u2) at (3.25,4.75);
                    
                    \coordinate (v0) at (-10,3.5);
                    \coordinate (v1) at (-7.5,3.5);
                    \coordinate (v2) at (-9,1);
                    
                    \coordinate (w0) at (-11.5,3.5);
                    \coordinate (w1) at (-11,-.5);
                    
                    \draw [draw=none,pattern=north west lines, pattern color=NavyBlue]
                        (u0) -- (u1) -- (u2) -- cycle;
                        
                    \draw [draw=none,pattern=north west lines, pattern color=NavyBlue]
                        (v0) -- (v1) -- (v2) -- cycle;
                        
                    \draw [draw=none,fill=gray!35] (x6) -- (x7) -- (y0) -- (y1) -- cycle;
                    \draw [draw=none,fill=gray!35] (y6) -- (y7) -- (z0) -- (z1) -- cycle;
                    \draw [draw=none,fill=gray!35] (z6) -- (z7) -- (w0) -- (v0) -- cycle;
                    \draw [draw=none,fill=gray!35] (v2) -- (w1) -- (x0) -- (x1) -- cycle;
                    
                    \draw [thick,fill=gray!35] (x4) -- (x5) -- (y2) -- (y3) -- (u2) -- cycle;
                    \draw [thick,fill=gray!35] (z2) -- (z3) -- (u1) -- (y4) -- (y5) -- cycle;
                    \draw [thick,fill=gray!35] (x2) -- (x3) -- (u0) -- (z4) -- (z5) -- (v1) -- cycle;
                    
                    \foreach \i [evaluate={\j={\i+1}}] in {0,...,6} {
                        \draw [ultra thick] (x\i) -- (x\j) (y\i) -- (y\j) (z\i) -- (z\j);}
                    
                    \foreach \i [evaluate={\j={mod(\i+1,3)}}] in {0,1,2} {
                        \draw [thick,dashed] (u\i) -- (u\j) (v\i) -- (v\j);}
                        
                    \draw [thick] (y1) -- (x6) (z1) -- (y6) (v2) -- (x1) (z6) -- (v0) (w0) -- (v0) (v2) -- (w1);
                        
                    \draw [ultra thick,Fuchsia,shorten <=2pt,shorten >=2pt,->]
                        ([xshift=-7pt,yshift=7pt]x1) -- ([xshift=-7pt,yshift=7pt]x2);
                    \draw [ultra thick,Fuchsia,shorten <=2pt,shorten >=2pt,<-]
                        ([xshift=7pt,yshift=-7pt]v1) -- ([xshift=7pt,yshift=-7pt]v2);
                    \draw [ultra thick,Fuchsia,shorten <=2pt,shorten >=2pt,->]
                        ([xshift=-7pt,yshift=7pt]x3) -- ([xshift=-7pt,yshift=7pt]x4);
                    \draw [ultra thick,Fuchsia,shorten <=2pt,shorten >=2pt,->]
                        ([xshift=7pt,yshift=-7pt]u0) -- ([xshift=7pt,yshift=-7pt]u2);
                    \draw [ultra thick,Fuchsia,shorten <=1pt,shorten >=1pt,->]
                        ([xshift=0pt,yshift=10pt]x5) -- ([xshift=0pt,yshift=10pt]x6);
                    \draw [ultra thick,Fuchsia,shorten <=1pt,shorten >=1pt,->]
                        ([xshift=0pt,yshift=-10pt]y2) -- ([xshift=0pt,yshift=-10pt]y1);
                    \draw [ultra thick,Fuchsia,shorten <=1pt,shorten >=1pt,->]
                        ([xshift=10pt,yshift=0pt]z1) -- ([xshift=10pt,yshift=0pt]z2);
                    \draw [ultra thick,Fuchsia,shorten <=1pt,shorten >=1pt,->]
                        ([xshift=-10pt,yshift=0pt]y6) -- ([xshift=-10pt,yshift=0pt]y5);
                    \draw [ultra thick,Fuchsia,shorten <=1pt,shorten >=1pt,->]
                        ([xshift=0pt,yshift=-10pt]z5) -- ([xshift=0pt,yshift=-10pt]z6);
                    \draw [ultra thick,Fuchsia,shorten <=1pt,shorten >=1pt,->]
                        ([xshift=0pt,yshift=10pt]v1) -- ([xshift=0pt,yshift=10pt]v0);
                    \draw [ultra thick,Fuchsia,shorten <=2pt,shorten >=2pt,->]
                        ([xshift=7pt,yshift=-7pt]z3) -- ([xshift=7pt,yshift=-7pt]z4);
                    \draw [ultra thick,Fuchsia,shorten <=2pt,shorten >=2pt,->]
                        ([xshift=-7pt,yshift=7pt]u1) -- ([xshift=-7pt,yshift=7pt]u0);
                    \draw [ultra thick,Fuchsia,shorten <=2pt,shorten >=2pt,->]
                        ([xshift=-7pt,yshift=-7pt]y4) -- ([xshift=-7pt,yshift=-7pt]y3);
                    \draw [ultra thick,Fuchsia,shorten <=2pt,shorten >=2pt,->]
                        ([xshift=7pt,yshift=7pt]u1) -- ([xshift=7pt,yshift=7pt]u2);
                    \draw [ultra thick,Fuchsia,shorten <=2pt,shorten >=2pt,->]
                        ([xshift=-7pt,yshift=-7pt]v0) -- ([xshift=-7pt,yshift=-7pt]v2);
                        
                    \begin{scope}[decoration={
                        markings,
                        mark=at position 0 with {\arrow[scale=1.5,>=stealth,xshift=2.5\pgflinewidth]{<}}}
                        ]
                        \draw [thick,postaction=decorate] (barycentric cs:u0=1,u1=1,u2=1) to [bend right=40] (5,.5)
                            node [right] {a cellular disc};
                        
                        \draw [thick,postaction=decorate] (barycentric cs:x1=1,x2=1,v1=1,v2=1) to [bend right=30] (0,-5.5)
                            node [right] {a $1$-handle};
                            
                        \draw [thick,postaction=decorate] (barycentric cs:x2=1,x3=1,u0=1,z4=1,z5=1,v1=1) to [bend left=15] (3.5,-1.5)
                            node [below] {a vertex disc};
                    \end{scope}
                    
                    \draw node at (-5,10) {\Large$\Sigma$};
                \end{tikzpicture}
            \end{subfigure}
            $\mathlarger{\mathlarger{\mathlarger{\xrightarrow{\phantom{aa}f\phantom{aa}}}}}$
            \begin{subfigure}[c]{11em}
                \centering
                \begin{tikzpicture}[every node/.style={circle,draw,fill=black,inner sep=1.5pt},scale=.4,rotate=10]
                    \draw [draw=none,pattern=north west lines, pattern color=NavyBlue]
                        (0,0) -- (-2,1.5) -- (-2,-1.5) -- cycle;
                    \draw [draw=none,pattern=north west lines, pattern color=NavyBlue]
                        (0,0) -- (2,1.5) -- (2,-1.5) -- cycle;
                        
                    \draw node (v2) at (0,0) {};
                    \draw node (v0) at (-2,1.5) {};
                    \draw node (v1) at (-2,-1.5) {};
                    \draw node (v3) at (2,1.5) {};
                    \draw node (v5) at (2,-1.5) {};
                    \draw node (v4) at (3.5,2.5) {};
                    \draw node (v6) at (3.5,-2.5) {};
                        
                    \begin{scope}[decoration={
                        markings,
                        mark=at position 0.5 with {\arrow{<}}}
                        ]
                        \draw [ultra thick,Fuchsia,postaction=decorate] (v1) -- (v0);
                        \draw [ultra thick,Fuchsia,postaction=decorate] (v0) -- (v2);
                        \draw [ultra thick,Fuchsia,postaction=decorate] (v2) -- (v1);
                        \draw [ultra thick,Fuchsia,postaction=decorate] (v2) -- (v3);
                        \draw [ultra thick,Fuchsia,postaction=decorate] (v5) -- (v2);
                        \draw [ultra thick,Fuchsia,postaction=decorate] (v6) -- (v5);
                        \draw [ultra thick,Fuchsia,postaction=decorate] (v3) -- (v4);
                        \draw [ultra thick,Fuchsia,postaction=decorate] (v5) -- (v3);
                    \end{scope}
                    
                    \draw node [rectangle,draw=none,fill=none] at (.5,-2) {\Large$X$};
                \end{tikzpicture}
            \end{subfigure}
            \caption{Part of a transverse incompressible admissible surface.}\label{fig:transv-srf}
        \end{figure}
    
\section{Letter-quasimorphisms}\label{sec:let-qm}

    Letter-quasimorphisms were introduced by Heuer \cite{heuer-raags} as a means of obtaining spectral gaps for $\scl$ in certain amalgams and right-angled Artin groups.
    
    We denote by $\mathcal{A}$ the set of \emph{alternating words} in the free group $F_2\coloneqq F(a,b)$:
    \begin{multline*}
        \mathcal{A}\coloneqq \left\{b_0a_1b_1\cdots a_\ell b_\ell a_{\ell+1}\in F_2\st{}
        a_1,\dots,a_\ell\in\left\{a^{\pm1}\right\},\right.\\
        \left.\:b_1,\dots,b_\ell\in\left\{b^{\pm1}\right\},\:b_0\in\left\{1,b^{\pm1}\right\},\:a_{\ell+1}\in\left\{1,a^{\pm1}\right\}\right\}.
    \end{multline*}
    
    \begin{defi}[Heuer \cite{heuer-raags}]\label{def:let-qm}
        A \emph{letter-quasimorphism} on a group $G$ is a map $\Phi:G\rightarrow\mathcal{A}$ subject to the following conditions:
        \begin{enumerate}
            \item $\Phi\left(g^{-1}\right)=\Phi\left(g\right)^{-1}$ for every $g\in G$, and
            \item For every $g,h\in G$, either
            \begin{enumerate}
                \item $\Phi\left(gh\right)=\Phi\left(g\right)\Phi\left(h\right)$ --- in this case, paths labelled by $\Phi(g)$, $\Phi(h)$, and $\Phi(gh)$ form tripods in the Cayley graph of $F(a,b)$, as shown in Figure \ref{fig:letter-qm-dg} --- or
                \item $\Phi(g)=u^{-1}x_1v$, $\Phi(h)=v^{-1}x_2w$, and $\Phi\left(gh\right)^{-1}=w^{-1}x_3u$ as reduced expressions, for some alternating words $u,v,w\in\mathcal{A}$, and letters $x_1,x_2,x_3\in\left\{a^{\pm1},b^{\pm1}\right\}$ such that either $x_1,x_2,x_3,x_1x_2x_3\in\left\{a^{\pm1}\right\}$, or $x_1,x_2,x_3,x_1x_2x_3\in\left\{b^{\pm1}\right\}$ --- this situation is depicted in Figure \ref{fig:letter-qm-ndg}, but it should be noted that, contrary to \ref{fig:letter-qm-dg}, this picture does not live in the Cayley graph of $F(a,b)$ but is merely a diagrammatic representation of the above equalities.
            \end{enumerate}
        \end{enumerate}
    \end{defi}
    
    \begin{figure}[htb]
        \centering
        \begin{subfigure}[c]{15em}
            \centering
            \begin{tikzpicture}[every node/.style={rectangle,draw=none,fill=none},rotate=80,scale=.8]
                \coordinate (o) at (0,0);
                \foreach \i in {0,...,2}{
                    \coordinate (x\i) at (\i*360/3: 3cm);
                    \coordinate (xpp\i) at ({\i*360/3+12}:3cm);
                    \coordinate (xmm\i) at ({\i*360/3-12}:3cm);
                    \coordinate (xp\i) at ({\i*360/3+2.5}:3cm);
                    \coordinate (yp\i) at ({\i*360/3+60}:.15cm);
                    \draw [thick] (x\i) -- (o);}
                    
                \draw [ultra thick,->] (xpp1) to [bend left=35] (xmm2);
                \draw [draw=none] (xpp1) -- (xmm2) node [midway] {$\Phi(g)$};
                \draw [ultra thick,->] (xpp2) to [bend left=35] (xmm0);
                \draw [draw=none] (xpp2) -- (xmm0) node [midway] {$\Phi(h)$};
                \draw [ultra thick,<-] (xpp0) to [bend left=35] (xmm1);
                \draw [draw=none] (xpp0) -- (xmm1) node [midway] {$\Phi(gh)$};
                    
                \draw [ultra thick,->,RoyalPurple] (yp1) -- (xp1)
                    node [midway,above left] {$u$};
                \draw [ultra thick,->,RoyalPurple] (yp2) -- (xp2)
                    node [midway,below left] {$v$};
                \draw [ultra thick,->,RoyalPurple] (yp0) -- (xp0)
                    node [midway,right] {$w$};
            \end{tikzpicture}
            \caption{Degenerate case.}\label{fig:letter-qm-dg}
        \end{subfigure}
        \qquad
        \begin{subfigure}[c]{15em}
            \centering
            \begin{tikzpicture}[every node/.style={rectangle,draw=none,fill=none},rotate=80,scale=.8]
                \foreach \i in {0,...,2}{
                    \coordinate (x\i) at (\i*360/3: 3cm);
                    \coordinate (xp\i) at ({\i*360/3+2.5}:3cm);
                    \coordinate (xpp\i) at ({\i*360/3+12}:3cm);
                    \coordinate (xmm\i) at ({\i*360/3-12}:3cm);
                    \coordinate (y\i) at (\i*360/3:.5cm);
                    \coordinate (yp\i) at ({\i*360/3+10}:.65cm);
                    \coordinate (ym\i) at ({\i*360/3-10}:.65cm);
                    \draw [thick] (x\i) -- (y\i);}
                
                \draw [thick] (y0) -- (y1) -- (y2) -- cycle;
                
                \draw [ultra thick,->,red,shorten <= 5pt, shorten >= 5pt] (yp1) -- (ym2)
                    node [midway,below] {$x_1$};
                \draw [ultra thick,->,red,shorten <= 5pt, shorten >= 5pt] (yp2) -- (ym0)
                    node [midway,right] {$x_2$};
                \draw [ultra thick,->,red,shorten <= 5pt, shorten >= 5pt] (yp0) -- (ym1)
                    node [midway,above left] {$x_3$};
                    
                \draw [ultra thick,->,RoyalPurple] (yp1) -- (xp1)
                    node [midway,above left] {$u$};
                \draw [ultra thick,->,RoyalPurple] (yp2) -- (xp2)
                    node [midway,below left] {$v$};
                \draw [ultra thick,->,RoyalPurple] (yp0) -- (xp0)
                    node [midway,right] {$w$};
                    
                \draw [ultra thick,->] (xpp1) to [bend left=25] (xmm2);
                \draw [draw=none] (xpp1) -- (xmm2) node [midway] {$\Phi(g)$};
                \draw [ultra thick,->] (xpp2) to [bend left=25] (xmm0);
                \draw [draw=none] (xpp2) -- (xmm0) node [midway] {$\Phi(h)$};
                \draw [ultra thick,<-] (xpp0) to [bend left=25] (xmm1);
                \draw [draw=none] (xpp0) -- (xmm1) node [midway] {$\Phi(gh)$};
            \end{tikzpicture}
            \caption{Non-degenerate case.}\label{fig:letter-qm-ndg}
        \end{subfigure}
        \caption{Letter-quasimorphism condition.}\label{fig:letter-qm}
    \end{figure}
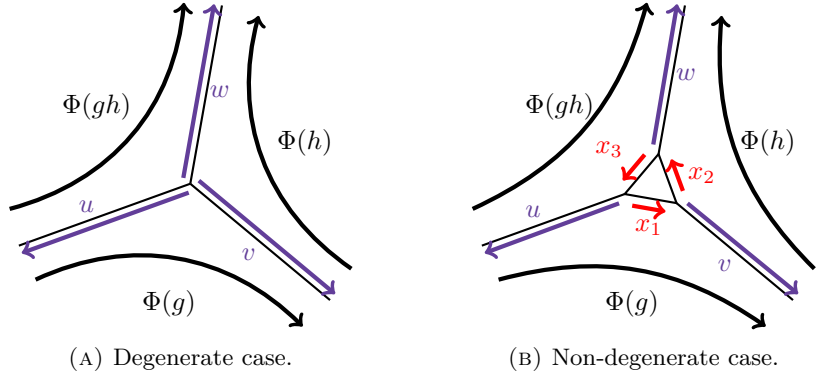
    
    \begin{rmk}\label{rk:let-qm-deg-case}
        Let $\Phi:G\rightarrow\mathcal{A}$ be a letter-quasimorphism and let $g,h\in G$ be such that $\Phi(gh)=\Phi(g)\Phi(h)$.
        Write $\Phi(g)=u^{-1}v$, $\Phi(h)=v^{-1}w$, and $\Phi(gh)^{-1}=w^{-1}u$ as reduced words, as in Figure \ref{fig:letter-qm-dg}.
        Then in fact one of $u$, $v$, and $w$ must be trivial.
        Indeed, otherwise, since $\Phi(g)$ is a reduced alternating word, if for example the first letter of $u$ is $a$ or $a^{-1}$, then the first letter of $v$ is $b$ or $b^{-1}$.
        Similarly, $\Phi(h)$ is an alternating word, so the first letter of $w$ must be $a$ or $a^{-1}$.
        But now both $u$ and $w$ start with $a$ or $a^{-1}$, which contradicts $\Phi(gh)$ being reduced and alternating.
    \end{rmk}
    
    We will need groups to have enough letter-quasimorphisms in the following sense:
    
    \begin{defi}\label{def:qrf}
        We will say that a group $G$ is \emph{quasi-residually free} if for every $g\in G$ with $\scl(g)<\infty$\footnote{The condition that $\scl(g)<\infty$ is equivalent to $g$ vanishing in $H_1\left(G;\mathbb{Q}\right)$, or equivalently, to some power of $g$ lying in the commutator subgroup.}, there is a letter-quasimorphism $\Phi:G\rightarrow\mathcal{A}$ such that $\Phi(g)\neq1$, and $\Phi\left(g^n\right)=\Phi(g)^n$ for all $n\in\mathbb{Z}$.
    \end{defi}
    
    \begin{ex}[Heuer \cite{heuer-raags}]
        The following groups are quasi-residually free:
        \begin{enumerate}
            \item Residually free groups \cite{heuer-raags}*{Example 4.2},
            \item Right-angled Artin groups \cite{heuer-raags}*{Theorem 7.2}.
        \end{enumerate}
    \end{ex}
    
    Heuer \cite{heuer-raags} showed that quasi-residually free groups have a strong spectral gap for $\scl$.
    To do this, he explains how one can use a letter-quasimorphism $G\rightarrow\mathcal{A}$ to construct a classical, real-valued quasimorphism $G\rightarrow\mathbb{R}$; he then uses Bavard duality to obtain lower bounds on $\scl$.
    
    We take a different route and show that a letter-quasimorphism gives rise to an angle structure on admissible surfaces, which can in turn be used to estimate the Euler characteristic.
    This approach is closer in spirit to Duncan and Howie's strong spectral gap in free groups \cite{duncan-howie}.
    
\section{Angle structures}\label{sec:ang-str}
    
    We will work with combinatorial $2$-dimensional CW-complexes, which we call \emph{$2$-complexes} for short.
    Combinatoriality means that the boundary of each cell has a subdivision for which the restriction of the attaching map to each subcell is a homeomorphism onto its image.
    This follows Gersten's terminology \cite{gersten}.
    
    \subsection{Definition and Gauß--Bonnet Formula}
    
    Wise \cite{wise:ang} introduced angle structures to study the interplay of geometric and topological properties of $2$-complexes.
    The idea is to assign a numerical angle to each corner of the $2$-complex of interest.
    
    Let $X$ be a $2$-complex.
    We call the $0$-cells of $X$ \emph{vertices}, its $1$-cells \emph{edges}, and its $2$-cells \emph{faces}.
    A \emph{corner} of a face $f$ at a vertex $v$ is an edge in the link $\lk_X(v)$ of $v$ corresponding to the face $f$.
    We denote by $\mathcal{C}(f)$ the set of corners of the face $f$ and by $\mathcal{C}(v)$ the set of corners at the vertex $v$.
    
    \begin{defi}[Wise \cite{wise:ang}]
        An \emph{angle structure} on $X$ is the assignment of a real number $\angle c$ to each corner $c$ in $X$.
        We then say that $X$ is an \emph{angled $2$-complex}.
    \end{defi}
        
    Note that any real angle value is allowed --- in particular, there is no restriction that angles be positive.
    
    Associated to an angle structure, there is a notion of curvature:
    \begin{itemize}
        \item Let $f$ be a face of $X$.
        The \emph{curvature} of $f$ is
        \[
            \kappa(f)\coloneqq 2\pi+\sum_{c\in\mathcal{C}(f)}\left(\angle c-\pi\right).
        \]
        \item Let $v$ be a vertex of $X$.
        The \emph{curvature} of $v$ is
        \[
            \kappa(v)\coloneqq 2\pi-\pi\cdot\chi\left(\lk_X(v)\right)-\sum_{c\in\mathcal{C}(v)}\angle c.
        \]
    \end{itemize}
    
    An angle structure on $X$ is said to be \emph{non-positively curved} if $\kappa(f)\leq0$ for every face $f$ of $X$ and $\kappa(v)\leq0$ for every vertex $v$ of $X$.
    
    The \emph{total curvature} of an angled $2$-complex $X$ is the quantity
    \[
        \kappa(X)\coloneqq\sum_{f\in F(X)}\kappa(f)+\sum_{v\in V(X)}\kappa(v),
    \]
    where $F(X)$ and $V(X)$ are the face set and vertex set of $X$, respectively.
    
    Angle structures arise naturally --- a prominent example is that of $\CAT(0)$ cell complexes \cite{bh}*{Chapter II.3}.
    The natural angle structure on a $\CAT(0)$ cell complex is non-positively curved.
    
    The relevance of angle structures comes from the \emph{Combinatorial Gauß--Bonnet Formula}, which constrains the topology of a $2$-complex depending on its total curvature.
    
    \begin{prop}[Combinatorial Gauß--Bonnet]\label{prop:gb}
        Let $X$ be an angled $2$-complex.
        Then
        \[
            2\pi\cdot\chi(X)=\kappa(X).
        \]
    \end{prop}
    \begin{proof}
        By definition, the total curvature $\kappa(X)$ is equal to
        \[
            \sum_{f\in F(X)}\left(2\pi+\sum_{c\in\mathcal{C}(f)}\left(\angle c-\pi\right)\right)+\sum_{v\in V(X)}\left(2\pi-\pi\cdot\chi\left(\lk_X(v)\right)-\sum_{c\in\mathcal{C}(v)}\angle c\right).
        \]
        Note that every corner appears as a corner of a unique face at a unique vertex, so that
        \begin{equation}
            \sum_{f\in F(X)}\sum_{c\in\mathcal{C}(f)}\angle c=\sum_{v\in V(X)}\sum_{c\in\mathcal{C}(v)}\angle c.\label{eq:sums-corners}
        \end{equation}
        Hence,
        \[
            \kappa(X)=2\pi\left(\sharp F(X)+\sharp V(X)\right)-\pi\sum_{f\in F(X)}\sharp\mathcal{C}(f)-\pi\sum_{v\in V(X)}\chi\left(\lk_X(v)\right).
        \]
        Now the link of a vertex $v$ is a graph, with vertices corresponding to half-edges of $X$ incident to $v$, and edges corresponding to corners of $X$ at $v$.
        It follows that
        \[
            \sum_{v\in V(X)}\chi\left(\lk_X(v)\right)=2\sharp E(X)-\sum_{v\in V(X)}\sharp\mathcal{C}(v).
        \]
        Similar to $\left(\ref{eq:sums-corners}\right)$, we have $\sum_f\sharp\mathcal{C}(f)=\sum_v\sharp\mathcal{C}(v)$.
        We finally obtain
        \[
            \kappa(X)=2\pi\left(\sharp F(X)-\sharp E(X)+\sharp V(X)\right)=2\pi\cdot\chi(X).\qedhere
        \]
    \end{proof}
    
    \subsection{Interior curvature of surfaces}
    
    We have defined angle structures on $2$-complexes, which are constructed by gluing a collection of discs to a graph.
    Each disc then has a notion of curvature as explained above.
    The objects of which we will want to estimate the curvature in the sequel will have a more natural decomposition into compact subsurfaces, rather than discs; this is our motivation for introducing the following notions.
    
    We consider a compact surface $\Lambda$ with nonempty boundary, with a cellulation whose vertex set $V(\Lambda)$ is contained in $\partial\Lambda$.
    Suppose that $\Lambda$ is equipped with an angle structure.
    \begin{itemize}
        \item The \emph{total angle} of a vertex $v$ of $\Lambda$ is
        \[
            \angle_{\tot}^\Lambda(v)\coloneqq\sum_{c\in\mathcal{C}(v)}\angle c.
        \]
        The reason why we include $\Lambda$ in the notation is that we will typically be considering a decomposition of a $2$-complex into subsurfaces, so that a vertex can be seen as belonging to one subsurface or another.
        \item The \emph{interior curvature} of $\Lambda$ is
        \[
            \kappa_{\intr}(\Lambda)\coloneqq2\pi\cdot\chi(\Lambda)+\sum_{v\in V(\Lambda)}\left(\angle_{\tot}^\Lambda(v)-\pi\right).
        \]
        Note in particular that if $\Lambda$ is a disc, then $\chi(\Lambda)=1$ and $\kappa_{\intr}(\Lambda)$ is the curvature of $\Lambda$ seen as a face.
    \end{itemize}
    
    We'll use the following elementary observation:
    
    \begin{lmm}\label{lmm:int-curv}
        Let $\Lambda$ be a compact surface equipped with an angle structure supported on a cellulation of $\Lambda$ with $V(\Lambda)\subseteq\partial\Lambda$.
        Then
        \[
            \kappa_{\intr}(\Lambda)=\sum_{f\in F(\Lambda)}\kappa(f),
        \]
        where $F(\Lambda)$ is the face set of $\Lambda$.
    \end{lmm}
    
    \begin{proof}
        Each vertex $v\in V(\Lambda)$ lies on $\partial\Lambda$, so its link is homeomorphic to a line segment, and we have
        \[
            \kappa(v)=2\pi-\pi\cdot\chi\left(\lk_\Lambda(v)\right)-\sum_{c\in\mathcal{C}(v)}\angle c=\pi-\angle_{\tot}^\Lambda(v).
        \]
        Hence, the Gauß--Bonnet Formula (Proposition \ref{prop:gb}) gives
        \[
            \sum_{f\in F(\Lambda)}\kappa(f)=2\pi\cdot\chi(\Lambda)+\sum_{v\in V(\Lambda)}\left(\angle_{\tot}^\Lambda(v)-\pi\right)=\kappa_{\intr}(\Lambda).\qedhere
        \]
    \end{proof}
    
    This says that, if $2$-cells --- which are normally discs --- are replaced with more general compact surfaces with nonempty boundary, then the interior curvature is the right analogue of the curvature of a face.
    
    \begin{cor}\label{cor:surf-dec}
        Let $X$ be an angled $2$-complex.
        Suppose that there is a finite collection $\left(\Lambda_i\right)_{i\in I}$ of compact surfaces with nonempty boundary cellularly embedded in $X$, such that
        \begin{itemize}
            \item Each face of $X$ is contained in a unique $\Lambda_i$, and
            \item Each vertex of $X$ is contained in $\partial\Lambda_i$ for some (possibly several) $i\in I$.
        \end{itemize}
        Then the total curvature of $X$ can be computed via
        \[
            \pushQED{\qed}
            \kappa(X)=\sum_{i\in I}\kappa_{\intr}\left(\Lambda_i\right)+\sum_{v\in V(X)}\kappa(v).\qedhere
            \popQED
        \]
    \end{cor}
    
\section{From letter-quasimorphisms to angle structures}\label{sec:lqm-ang}
    
    We now reach the core of this paper.
    The goal is to prove the following:
    
    \begin{mthm}{A}\label{thm:exist-ang-str}
        Let $X$ be a connected $2$-complex and let $g\in\pi_1X\smallsetminus1$.
        Assume that there is a letter-quasimorphism $\Phi:\pi_1X\rightarrow\mathcal{A}$ with $\Phi(g)\neq1$ and $\Phi\left(g^n\right)=\Phi(g)^n$ for all $n\in\mathbb{Z}$.
        
        Then given a monotone, incompressible, disc- and sphere-free admissible surface $f:(\Sigma,\partial\Sigma)\rightarrow(X,g)$ with $f_*\left[\partial\Sigma\right]=n(\Sigma)\left[S^1\right]$ for some $n(\Sigma)\in\mathbb{N}_{\geq1}$, there is an angle structure on $\Sigma$ whose total curvature satisfies
        \[
            \kappa\left(\Sigma\right)\leq-2\pi\cdot n(\Sigma).
        \]
    \end{mthm}
    
    Before entering the proof of Theorem \ref{thm:exist-ang-str}, we explain its implications for $\scl$:
    
    \begin{cor}[Heuer \cite{heuer-raags}*{Theorem 4.7}]
        Let $G$ be a group and $g\in G\smallsetminus1$.
        Assume that there is a letter-quasimorphism $\Phi:G\rightarrow\mathcal{A}$ with $\Phi(g)\neq1$ and $\Phi\left(g^n\right)=\Phi(g)^n$ for all $n\in\mathbb{Z}$.
        
        Then $\scl(g)\geq\frac{1}{2}$.
    \end{cor}
    
    \begin{proof}
        Fix $X$ a connected $2$-complex with $\pi_1X=G$ (for example, $X$ can be a presentation $2$-complex of $G$).
        Proposition \ref{prop:monot-incompr} says that $\scl$ can be computed as the infimum of $\frac{-\chi^-(\Sigma)}{2n(\Sigma)}$ over all disc- and sphere-free, incompressible, monotone admissible surfaces.
        For any such admissible surface $\Sigma$, Theorem \ref{thm:exist-ang-str} gives an angle structure on $\Sigma$ with a bound on the total curvature.
        Now the Gauß--Bonnet Formula (Proposition \ref{prop:gb}) translates the bound on $\kappa\left(\Sigma\right)$ into a bound on $\chi\left(\Sigma\right)$, and one obtains
        \[
            -\chi^-(\Sigma)\geq-\chi(\Sigma)=-\frac{1}{2\pi}\kappa\left(\Sigma\right)\geq n(\Sigma).
        \]
        After passing to the infimum, this implies that $\scl(g)\geq\frac{1}{2}$.
    \end{proof}
    
    \begin{cor}[Heuer \cite{heuer-raags}]
        Every quasi-residually free group has a strong spectral gap for $\scl$.\qed
    \end{cor}
    
    The rest of this section is devoted to proving the main theorem.
    
    \begin{proof}[Proof of Theorem \ref{thm:exist-ang-str}]
    
        Note that $\Phi(g)$ is an alternating word in $F_2=F(a,b)$.
        Furthermore, the assumption that $\Phi\left(g^n\right)=\Phi\left(g\right)^n$ for all $n$, and the fact that each $\Phi\left(g^n\right)$ is an alternating word implies that $\Phi(g)$ is cyclically reduced, and cannot start and end in the same letter.
        Let us assume to fix notations that it starts with $a$ or $a^{-1}$ and ends with $b$ or $b^{-1}$, so that
        \[
            \Phi(g)=a_1b_1\cdots a_\ell b_\ell,
        \]
        with $a_i\in\left\{a^{\pm1}\right\}$ and $b_i\in\left\{b^{\pm1}\right\}$.
    
        Start with $f:(\Sigma,\partial\Sigma)\rightarrow(X,g)$ a disc- and sphere-free, incompressible, monotone admissible surface with $f_*\left[\partial\Sigma\right]=n(\Sigma)\left[S^1\right]$ for some $n(\Sigma)\in\mathbb{N}_{\geq1}$.
        
        After contracting a spanning tree in the $1$-skeleton, one may assume that $X$ has only one vertex.
        One may also subdivide the faces of $X$ to ensure that they are all triangles.
        After applying Proposition \ref{prop:transv}, we can further assume that $f$ is transverse.
        
    \def\spr{5}
    \def\angred{45}
    \def\anggrn{75}
    \subsection{Subdividing \texorpdfstring{$1$}{1}-handles via the letter-quasimorphism}
        
        Each $1$-handle $\mathcal{H}$ in $\Sigma$ is a trivial $I$-bundle over an edge $e$ of $X$, which is a loop and represents an element $g_\mathcal{H}\in\pi_1X$.
        We consider the image of $g_\mathcal{H}$ under our letter-quasimorphism $\Phi:\pi_1X\rightarrow\mathcal{A}$; this is an alternating word in $F_2$, and we write for example
        \begin{equation}\label{eq:phigh}
            \Phi\left(g_\mathcal{H}\right)=\alpha_1\beta_1\cdots\alpha_k\beta_k,
        \end{equation}
        with $\alpha_i\in\left\{a^{\pm1}\right\}$ and $\beta_i\in\left\{b^{\pm1}\right\}$.
        Note that $\Phi\left(g_\mathcal{H}\right)$ might have odd length and might not be cyclically reduced, but the construction remains the same.
        The element $\Phi\left(g_\mathcal{H}\right)\in F_2$ is represented by an immersed loop $\gamma_{\Phi\left(g_\mathcal{H}\right)}:S^1\looparrowright B_2$, where $B_2\coloneqq S^1_a\vee S^1_b$ is a bouquet of two oriented circles labelled by $a$ and $b$.
        
        The $1$-handle $\mathcal{H}$ being a trivial $I$-bundle over $e$ means that it is homeomorphic to the product $I\times e$ (where $I\coloneqq[0,1]$), and the map $\mathcal{H}\rightarrow e$ is given by projection onto the second coordinate.
        Now pick an orientation-preserving homeomorphism $e\cong S^1$ and consider the composition
        \[
            \mathcal{H}=I\times e\xrightarrow{\textnormal{proj}_2}e\cong S^1\xrightarrow{\gamma_{\Phi\left(g_\mathcal{H}\right)}}B_2=S^1_a\vee S^1_b.
        \]
        This map is transverse, and defines a decomposition of $\mathcal{H}$ into $2k$ $1$-handles (if $\Phi\left(g_\mathcal{H}\right)$ has even length as in (\ref{eq:phigh})) mapping successively to $\alpha_1,\beta_1,\dots,\alpha_k,\beta_k$.
        To distinguish those $1$-handles from the original ones, we call them \emph{stripes}.
        Stripes can be divided into two types:
        \begin{itemize}
            \item \emph{$a$-stripes}, with image contained in $S^1_a$, and
            \item \emph{$b$-stripes}, with image contained in $S^1_b$.
        \end{itemize}
        We'll depict those two types of stripes with two different colours (red and green).
        The map from each stripe to $S^1_a$ or $S^1_b$ will be encoded in pictures by arrows parallel to the base edge of the $I$-bundle, indicating the positive direction of $S^1_a$ or $S^1_b$ --- see Figure \ref{fig:new-handle}.
        
        \begin{figure}[hbt]
            \centering
            \begin{subfigure}[c]{13em}
                \centering
                \begin{tikzpicture}[every node/.style={rectangle,draw=none,fill=none},rotate=10,xscale=1,yscale=.4]
                    \coordinate (a0) at (0,0);
                    \coordinate (a1) at (3,0);
                    \coordinate (b0) at (0,4);
                    \coordinate (b1) at (3,4);
                    \draw [thick] (a0) -- (a1) -- (b1) -- (b0) -- cycle;
                    \draw [ultra thick,Fuchsia,shorten <=2pt,shorten >=2pt,->]
                        ([yshift=5pt]a0) -- ([yshift=5pt]a1)
                        node [midway,below,inner sep=7pt] {$g_\mathcal{H}$};
                    \draw [ultra thick,Fuchsia,shorten <=2pt,shorten >=2pt,->]
                        ([yshift=-5pt]b0) -- ([yshift=-5pt]b1);
                    \node at (barycentric cs:a0=1,a1=1,b0=1,b1=1) {\large$\mathcal{H}$};
                \end{tikzpicture}
            \end{subfigure}
            $\mathlarger{\mathlarger{\mathlarger{\mathlarger{\rightsquigarrow}}}}$
            \begin{subfigure}[c]{13em}
                \centering
                \begin{tikzpicture}[every node/.style={rectangle,draw=none,fill=none},rotate=10,xscale=1,yscale=.4]
                    \foreach \i in {0,1,...,6}{
                        \coordinate (a\i) at ({\i/6*3},0);
                        \coordinate (b\i) at ({\i/6*3},4);}
                    \foreach \x [evaluate={\i={int(2*(\x-1))}},evaluate={\j={int(\i+1)}}] in {1,2,3}{
                        \draw [draw=none,pattern={Lines[angle=\angred,distance=\spr]},pattern color=red]
                            (a\i) -- (a\j) -- (b\j) -- (b\i) -- cycle;}
                    \foreach \x [evaluate={\i={int(2*\x-1)}},evaluate={\j={int(\i+1)}}] in {1,2,3}{
                        \draw [draw=none,pattern={Lines[angle=\anggrn,distance=\spr]},pattern color=PineGreen]
                            (a\i) -- (a\j) -- (b\j) -- (b\i) -- cycle;}
                    \foreach \i in {0,1,...,6}{
                        \draw [thick] (a\i) -- (b\i);}
                    \draw [thick] (a0) -- (a6) (b0) -- (b6);
                    
                    \draw [ultra thick,PineGreen,shorten <=1pt,shorten >=1pt,<-]
                        ([yshift=5pt]a1) -- ([yshift=5pt]a2) node[midway,above,yshift=.3pt,circle,draw=none,fill=white,inner sep=.15pt] {$b$};
                    \draw [ultra thick,PineGreen,shorten <=1pt,shorten >=1pt,<-]
                        ([yshift=-5pt]b1) -- ([yshift=-5pt]b2);
                        
                    \draw [ultra thick,PineGreen,shorten <=1pt,shorten >=1pt,<-]
                        ([yshift=5pt]a3) -- ([yshift=5pt]a4) node[midway,above,yshift=.3pt,circle,draw=none,fill=white,inner sep=.15pt] {$b$};
                    \draw [ultra thick,PineGreen,shorten <=1pt,shorten >=1pt,<-]
                        ([yshift=-5pt]b3) -- ([yshift=-5pt]b4);
                        
                    \draw [ultra thick,PineGreen,shorten <=1pt,shorten >=1pt,->]
                        ([yshift=5pt]a5) -- ([yshift=5pt]a6) node[midway,above,yshift=.3pt,circle,draw=none,fill=white,inner sep=.15pt] {$b$};
                    \draw [ultra thick,PineGreen,shorten <=1pt,shorten >=1pt,->]
                        ([yshift=-5pt]b5) -- ([yshift=-5pt]b6);
                        
                    \draw [ultra thick,red,shorten <=1pt,shorten >=1pt,->]
                        ([yshift=5pt]a0) -- ([yshift=5pt]a1) node[midway,above,yshift=.3pt,circle,draw=none,fill=white,inner sep=.15pt] {$a$};
                    \draw [ultra thick,red,shorten <=1pt,shorten >=1pt,->]
                        ([yshift=-5pt]b0) -- ([yshift=-5pt]b1);
                        
                    \draw [ultra thick,red,shorten <=1pt,shorten >=1pt,<-]
                        ([yshift=5pt]a2) -- ([yshift=5pt]a3) node[midway,above,yshift=.3pt,circle,draw=none,fill=white,inner sep=.15pt] {$a$};
                    \draw [ultra thick,red,shorten <=1pt,shorten >=1pt,<-]
                        ([yshift=-5pt]b2) -- ([yshift=-5pt]b3);
                        
                    \draw [ultra thick,red,shorten <=1pt,shorten >=1pt,->]
                        ([yshift=5pt]a4) -- ([yshift=5pt]a5) node[midway,above,yshift=.3pt,circle,draw=none,fill=white,inner sep=.15pt] {$a$};
                    \draw [ultra thick,red,shorten <=1pt,shorten >=1pt,->]
                        ([yshift=-5pt]b4) -- ([yshift=-5pt]b5);
                            
                    \draw [ultra thick,shorten <=2pt,shorten >=2pt,->]
                        ([yshift=-7pt]a0) -- ([yshift=-7pt]a6)
                        node [midway,below] {$\Phi\left(g_\mathcal{H}\right)$};
                \end{tikzpicture}
            \end{subfigure}
            \caption{Decomposition of a $1$-handle into stripes via the letter-quasimorphism. Here, $\Phi\left(g_\mathcal{H}\right)=ab^{-1}a^{-1}b^{-1}ab$.}\label{fig:new-handle}
        \end{figure}
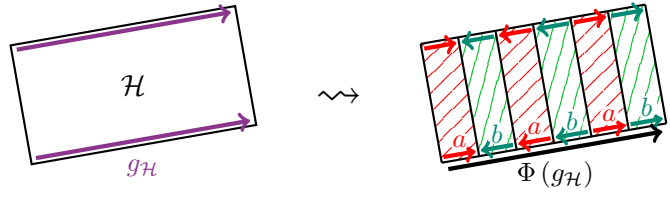
        
    \subsection{Extension to cellular discs}\label{subsec:ext-cell}
        
        The above construction defines a new cellular structure on the $1$-handles of $\Sigma$, and we now want to do something similar to cellular discs.
        At the end of the construction, $\Sigma$ will be decomposed into its preexisting vertex discs, $a$-regions (extending $a$-stripes), and $b$-regions (extending $b$-stripes).
        
        Each cellular disc $\mathcal{D}$ maps homeomorphically to a (triangular) $2$-cell of $X$, and has three sides mapping under $f$ to edges representing elements $g_1,g_2,g_3\in\pi_1X$.
        The presence of the $2$-cell means that $g_3=\left(g_1g_2\right)^{-1}$.
        Now the three $1$-handles incident to $\mathcal{D}$ have been subdivided as explained above into parallel stripes labelled by $a^{\pm1}$ and $b^{\pm1}$, and the concatenation of those labels along each of the three edges of the triangle are given by $\Phi\left(g_1\right)$, $\Phi\left(g_2\right)$, and $\Phi\left(g_1g_2\right)^{-1}$.
        By the letter-quasimorphism condition, they satisfy one of the patterns of Figure \ref{fig:letter-qm}: we can write
        \[
            \Phi\left(g_1\right)=u^{-1}x_1v,\quad\Phi\left(g_2\right)=v^{-1}x_2w,\quad\Phi\left(g_1g_2\right)^{-1}=w^{-1}x_3u,
        \]
        with either $x_1=x_2=x_3=1$, or $x_1,x_2,x_3,x_1x_2x_3\in\left\{a^{\pm1}\right\}$, or $x_1,x_2,x_3,x_1x_2x_3\in\left\{b^{\pm1}\right\}$.
        Moreover, Remark \ref{rk:let-qm-deg-case} says that, in the degenerate case (when $x_1=x_2=x_3=1$), we have $u=1$ or $v=1$ or $w=1$ --- we will assume that $w=1$ to fix notations.
        
        As illustrated by Figure \ref{fig:new-discs}, the boundary of $\mathcal{D}$ has two consecutive sections labelled respectively by $u$ and $u^{-1}$, and these sections lie on different sides of $\mathcal{D}$.
        We can connect the vertices of these two sections and extend the $a$- and $b$-stripes across part of $\mathcal{D}$.
        We then apply the same process to $v$ and $w$.
        The new stripes that we have constructed inside $\mathcal{D}$ are called $a$- or $b$-stripes depending on whether they extend $a$- or $b$-stripes.
        Note that some of the new stripes, at the vertices of the cellular disc, are triangular (one of the fibres of the $I$-bundle is collapsed to a point), but we still call them stripes.
        
        In the degenerate case (Figure \ref{fig:new-discs-dg}), the new stripes now fill $\mathcal{D}$.
        In the non-degenerate case (Figure \ref{fig:new-discs-ndg}), we are left with one unfilled hexagon inside $\mathcal{D}$; this hexagon has three sides lying on $\partial\mathcal{D}$ and mapping either all to $S^1_a$ or all to $S^1_b$.
        We declare the hexagon to be of type $a$ or $b$ accordingly.
        
        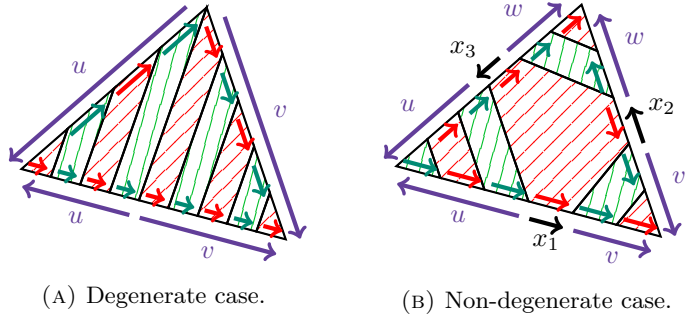
\begin{figure}[hbt]
            \centering
            \begin{subfigure}[c]{14em}
                \centering
                \begin{tikzpicture}[every node/.style={rectangle,draw=none,fill=none},rotate=-15,scale=.9]
                    \coordinate (a) at (0,3);
                    \coordinate (b) at (-2,0);
                    \coordinate (c) at (2,0);
                    
                    \foreach \i in {0,...,9}{
                        \coordinate (bc\i) at ({-2+4*\i/9},0);}
                    \foreach \i in {0,...,4}{
                        \coordinate (ba\i) at ({-2+2*\i/4},{3*\i/4});}
                    \foreach \i in {0,...,5}{
                        \coordinate (ac\i) at ({2*\i/5},{3-3*\i/5});}
                        
                    \foreach \i [evaluate={\j={int(\i-1)}}] in {1,3}{
                        \draw [pattern={Lines[angle=\angred,distance=\spr]},pattern color=red]
                            (bc\i) -- (ba\i) -- (ba\j) -- (bc\j) -- cycle;
                        \draw [thick]
                            (bc\i) -- (ba\i) -- (ba\j) -- (bc\j) -- cycle;
                        \draw [ultra thick,red,shorten <=2pt,shorten >=2pt,->]
                            ([yshift=+3pt]bc\j) -- ([yshift=+3pt]bc\i);
                        \ifnum\j>1
                            \draw [ultra thick,red,shorten <=3pt,shorten >=3pt,->]
                                ([yshift=-5pt]ba\j) -- ([yshift=-5pt]ba\i);
                        \fi}
                            
                    \foreach \i [evaluate={\j={int(\i-1)}}] in {2,4}{
                        \draw [pattern={Lines[angle=\anggrn,distance=\spr]},pattern color=PineGreen]
                            (bc\i) -- (ba\i) -- (ba\j) -- (bc\j) -- cycle;
                        \draw [thick]
                            (bc\i) -- (ba\i) -- (ba\j) -- (bc\j) -- cycle;
                        \draw [ultra thick,PineGreen,shorten <=2pt,shorten >=2pt,->]
                            ([yshift=+3pt]bc\j) -- ([yshift=+3pt]bc\i);
                        \draw [ultra thick,PineGreen,shorten <=3pt,shorten >=3pt,->]
                            ([yshift=-5pt]ba\j) -- ([yshift=-5pt]ba\i);}
                            
                    \foreach \i [evaluate={\j={int(\i-1)}},evaluate={\ip={int(\i+4)}},evaluate={\jp={int(\j+4)}}] in {1,3,5}{
                        \draw [pattern={Lines[angle=\angred,distance=\spr]},pattern color=red]
                            (bc\ip) -- (ac\i) -- (ac\j) -- (bc\jp) -- cycle;
                        \draw [thick]
                            (bc\ip) -- (ac\i) -- (ac\j) -- (bc\jp) -- cycle;
                        \draw [ultra thick,red,shorten <=2pt,shorten >=2pt,->]
                            ([yshift=+3pt]bc\jp) -- ([yshift=+3pt]bc\ip);
                        \ifnum\i<5
                            \draw [ultra thick,red,shorten <=3pt,shorten >=3pt,->]
                                ([yshift=-5pt]ac\j) -- ([yshift=-5pt]ac\i);
                        \fi}
                        
                    \foreach \i [evaluate={\j={int(\i-1)}},evaluate={\ip={int(\i+4)}},evaluate={\jp={int(\j+4)}}] in {2,4}{
                        \draw [pattern={Lines[angle=\anggrn,distance=\spr]},pattern color=PineGreen]
                            (bc\ip) -- (ac\i) -- (ac\j) -- (bc\jp) -- cycle;
                        \draw [thick]
                            (bc\ip) -- (ac\i) -- (ac\j) -- (bc\jp) -- cycle;
                        \draw [ultra thick,PineGreen,shorten <=2pt,shorten >=2pt,->]
                            ([yshift=+3pt]bc\jp) -- ([yshift=+3pt]bc\ip);
                        \draw [ultra thick,PineGreen,shorten <=3pt,shorten >=3pt,->]
                            ([yshift=-5pt]ac\j) -- ([yshift=-5pt]ac\i);}
                    
                    \draw [ultra thick,shorten <=2pt,shorten >=2pt,<-,RoyalPurple] ([xshift=-7pt]b) -- ([xshift=-7pt]a)
                        node [midway,above left] {$u$};
                    \draw [ultra thick,shorten <=2pt,shorten >=2pt,->,RoyalPurple] ([xshift=7pt]a) -- ([xshift=7pt]c)
                        node [midway,above right] {$v$};
                    \draw [ultra thick,shorten <=2pt,shorten >=2pt,<-,RoyalPurple] ([yshift=-7pt]b) -- ([yshift=-7pt]bc4)
                        node [midway,below] {$u$};
                    \draw [ultra thick,shorten <=2pt,shorten >=2pt,->,RoyalPurple] ([yshift=-7pt]bc4) -- ([yshift=-7pt]c)
                        node [midway,below] {$v$};
                \end{tikzpicture}
                \caption{Degenerate case.}\label{fig:new-discs-dg}
            \end{subfigure}
            \begin{subfigure}[c]{14em}
                \centering
                \begin{tikzpicture}[every node/.style={rectangle,draw=none,fill=none},rotate=-15,scale=.9]
                    \coordinate (a) at (0,3);
                    \coordinate (b) at (-2,0);
                    \coordinate (c) at (2,0);
                    
                    \foreach \i in {0,...,6}{
                        \coordinate (bc\i) at ({-2+4*\i/6},0);
                        \coordinate (ba\i) at ({-2+2*\i/6},{3*\i/6});}
                    \foreach \i in {0,...,5}{
                        \coordinate (ac\i) at ({2*\i/5},{3-3*\i/5});}
                        
                    \foreach \i [evaluate={\j={int(\i-1)}}] in {1,3}{
                        \draw [pattern={Lines[angle=\anggrn,distance=\spr]},pattern color=PineGreen]
                            (bc\i) -- (ba\i) -- (ba\j) -- (bc\j) -- cycle;
                        \draw [thick]
                            (bc\i) -- (ba\i) -- (ba\j) -- (bc\j) -- cycle;}
                    \draw [pattern={Lines[angle=\angred,distance=\spr]},pattern color=red]
                        (bc2) -- (ba2) -- (ba1) -- (bc1) -- cycle;
                    \draw [thick]
                        (bc2) -- (ba2) -- (ba1) -- (bc1) -- cycle;
                    \draw [pattern={Lines[angle=\angred,distance=\spr]},pattern color=red]
                        (c) -- (bc5) -- (ac4) -- cycle;
                    \draw [thick]
                        (c) -- (bc5) -- (ac4) -- cycle;
                    \draw [pattern={Lines[angle=\angred,distance=\spr]},pattern color=red]
                        (a) -- (ac1) -- (ba5) -- cycle;
                    \draw [thick]
                        (a) -- (ac1) -- (ba5) -- cycle;
                    \draw [pattern={Lines[angle=\anggrn,distance=\spr]},pattern color=PineGreen]
                        (bc4) -- (bc5) -- (ac4) -- (ac3) -- cycle;
                    \draw [thick]
                        (bc4) -- (bc5) -- (ac4) -- (ac3) -- cycle;
                    \draw [pattern={Lines[angle=\anggrn,distance=\spr]},pattern color=PineGreen]
                        (ba4) -- (ba5) -- (ac1) -- (ac2) -- cycle;
                    \draw [thick]
                        (ba4) -- (ba5) -- (ac1) -- (ac2) -- cycle;
                    \draw [pattern={Lines[angle=\angred,distance=\spr]},pattern color=red]
                        (bc3) -- (bc4) -- (ac3) -- (ac2) -- (ba4) -- (ba3) -- cycle;
                    \draw [thick]
                        (bc3) -- (bc4) -- (ac3) -- (ac2) -- (ba4) -- (ba3) -- cycle;
                        
                    \foreach \i [evaluate={\j={int(\i-1)}},evaluate={\k={int(\i+1)}}] in {1,3,5}{
                        \draw [ultra thick,PineGreen,shorten <=2pt,shorten >=2pt,->]
                            ([yshift=+3pt]bc\j) -- ([yshift=+3pt]bc\i);
                        \draw [ultra thick,red,shorten <=2pt,shorten >=2pt,->]
                            ([yshift=+3pt]bc\i) -- ([yshift=+3pt]bc\k);}
                        
                    \foreach \i [evaluate={\j={int(\i-1)}},evaluate={\k={int(\i+1)}}] in {1,3,5}{
                        \ifnum\i>1
                            \draw [ultra thick,PineGreen,shorten <=3pt,shorten >=3pt,->]
                                ([xshift=+3pt,yshift=-2pt]ba\j) -- ([xshift=+3pt,yshift=-2pt]ba\i)
                        \fi;
                        \draw [ultra thick,red,shorten <=3pt,shorten >=3pt,->]
                            ([xshift=+3pt,yshift=-2pt]ba\i) -- ([xshift=+3pt,yshift=-2pt]ba\k);}
                            
                    \draw [ultra thick,PineGreen,shorten <=3pt,shorten >=3pt,<-]
                        ([xshift=-3pt,yshift=-2pt]ac1) -- ([xshift=-3pt,yshift=-2pt]ac2);
                    \draw [ultra thick,PineGreen,shorten <=3pt,shorten >=3pt,->]
                        ([xshift=-3pt,yshift=-2pt]ac3) -- ([xshift=-3pt,yshift=-2pt]ac4);
                    \draw [ultra thick,red,shorten <=3pt,shorten >=3pt,->]
                        ([xshift=-3pt,yshift=-2pt]ac2) -- ([xshift=-3pt,yshift=-2pt]ac3);
                    
                    \draw [ultra thick,shorten <=2pt,shorten >=2pt,->,RoyalPurple] ([xshift=-7pt]ba3) -- ([xshift=-7pt]b)
                        node [midway,above left] {$u$};
                    \draw [ultra thick,shorten <=2pt,shorten >=2pt,<-] ([xshift=-7pt]ba3) -- ([xshift=-7pt]ba4)
                        node [midway,above left] {$x_3$};
                    \draw [ultra thick,shorten <=2pt,shorten >=2pt,->,RoyalPurple] ([xshift=-7pt]ba4) -- ([xshift=-7pt]a)
                        node [midway,above left] {$w$};
                    \draw [ultra thick,shorten <=2pt,shorten >=2pt,->,RoyalPurple] ([xshift=7pt]ac2) -- ([xshift=7pt]a)
                        node [midway,above right] {$w$};
                    \draw [ultra thick,shorten <=2pt,shorten >=2pt,<-] ([xshift=7pt]ac2) -- ([xshift=7pt]ac3)
                        node [midway,above right] {$x_2$};
                    \draw [ultra thick,shorten <=2pt,shorten >=2pt,->,RoyalPurple] ([xshift=7pt]ac3) -- ([xshift=7pt]c)
                        node [midway,above right] {$v$};
                    \draw [ultra thick,shorten <=2pt,shorten >=2pt,->,RoyalPurple] ([yshift=-7pt]bc3) -- ([yshift=-7pt]b)
                        node [midway,below] {$u$};
                    \draw [ultra thick,shorten <=2pt,shorten >=2pt,<-] ([yshift=-7pt]bc4) -- ([yshift=-7pt]bc3)
                        node [midway,below] {$x_1$};
                    \draw [ultra thick,shorten <=2pt,shorten >=2pt,->,RoyalPurple] ([yshift=-7pt]bc4) -- ([yshift=-7pt]c)
                        node [midway,below] {$v$};
                \end{tikzpicture}
                \caption{Non-degenerate case.}\label{fig:new-discs-ndg}
            \end{subfigure}
            \caption{Extension of stripes across cellular discs.}\label{fig:new-discs}
        \end{figure}
        
        This construction yields a new cellular structure on $\Sigma$, with a decomposition into vertex discs (which remain unchanged), $a$- and $b$-stripes, and $a$- and $b$-hexagons --- see Figure \ref{fig:zebra}.
        We call this data a \emph{stripe pattern}.
        
        It is important to note that the new cellular structure does not correspond to any transverse map to a bouquet of two circles: such a map can be defined on $a$- and $b$-stripes as explained above, but it cannot be extended to hexagons since their boundary maps to a non-trivial element of $F_2$ (either $a^{\pm1}$ or $b^{\pm1}$).
        
        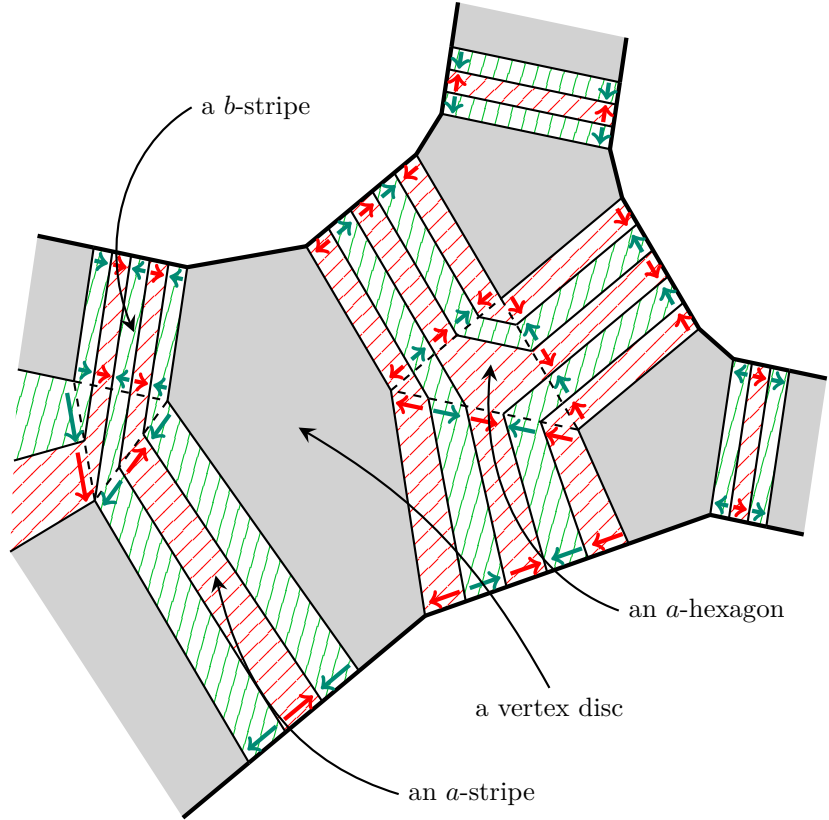
\begin{figure}[htb]
            \centering
            \begin{tikzpicture}[every node/.style={rectangle,draw=none,fill=none},xscale=.5,yscale=.6,rotate=-10]
                \coordinate (x0) at (-5.5,-5.5);
                \coordinate (x1) at (-4,-4);
                \coordinate (x2) at (-1.5,-1.5);
                \coordinate (x3) at (0,0);
                \coordinate (x4) at (5,2.5);
                \coordinate (x5) at (7,3.5);
                \coordinate (x6) at (8.5,3.5);
                \coordinate (x7) at (9.5,3.5);
                
                \coordinate (y0) at (9.5,7);
                \coordinate (y1) at (8.5,7);
                \coordinate (y2) at (7,7);
                \coordinate (y3) at (6,7.5);
                \coordinate (y4) at (3.5,10);
                \coordinate (y5) at (3,11);
                \coordinate (y6) at (3,12.5);
                \coordinate (y7) at (3,13.5);
                
                \coordinate (z0) at (-1.5,13.5);
                \coordinate (z1) at (-1.5,12.5);
                \coordinate (z2) at (-1.5,11);
                \coordinate (z3) at (-2,10);
                \coordinate (z4) at (-4.5,7.5);
                \coordinate (z5) at (-7.5,6.5);
                \coordinate (z6) at (-10,6.5);
                \coordinate (z7) at (-11.5,6.5);
                    
                \coordinate (u0) at (-1.75,4.75);
                \coordinate (u1) at (.75,7.25);
                \coordinate (u2) at (3.25,4.75);
                
                \coordinate (v0) at (-10,3.5);
                \coordinate (v1) at (-7.5,3.5);
                \coordinate (v2) at (-9,1);
                
                \coordinate (w0) at (-11.5,3.5);
                \coordinate (w1) at (-11,-.5);
                
                \foreach \i in {1,...,4}{
                    \coordinate (x34\i) at ({\i},{.5*\i});
                    \coordinate (u02\i) at ({-1.75+\i},{4.75});
                    \coordinate (z34\i) at ({-2-.5*\i},{10-.5*\i});
                    \coordinate (u10\i) at ({.75-.5*\i},{7.25-.5*\i});
                    \coordinate (y34\i) at ({6-.5*\i},{7.5+.5*\i});
                    \coordinate (u21\i) at ({3.25-.5*\i},{4.75+.5*\i});
                    \coordinate (z56\i) at ({-7.5-.5*\i},{6.5});
                    \coordinate (v10\i) at ({-7.5-.5*\i},{3.5});}
                    
                \foreach \i in {1,2}{
                    \coordinate (x56\i) at ({7+.5*\i},{3.5});
                    \coordinate (y21\i) at ({7+.5*\i},{7});
                    \coordinate (y56\i) at ({3},{11+.5*\i});
                    \coordinate (z21\i) at ({-1.5},{11+.5*\i});
                    \coordinate (x12\i) at ({-4+2.5*\i/3},{-4+2.5*\i/3});
                    \coordinate (v12\i) at ({-9+.5*\i},{1+2.5*\i/3});}
                    
                \coordinate (w011) at (-11.25,1.5);
                \coordinate (v021) at (-9.5,2.25);
                
                \draw [draw=none,pattern={Lines[angle=\anggrn,distance=\spr]},pattern color=PineGreen]
                    (z1) -- (z212) -- (y562) -- (y6) -- cycle;
                \draw [draw=none,pattern={Lines[angle=\anggrn,distance=\spr]},pattern color=PineGreen]
                    (z2) -- (z211) -- (y561) -- (y5) -- cycle;
                \draw [draw=none,pattern={Lines[angle=\anggrn,distance=\spr]},pattern color=PineGreen]
                    (z341) -- (z342) -- (u102) -- (u213) -- (y343) -- (y344) -- (u214) -- (u101) -- cycle;
                \draw [draw=none,pattern={Lines[angle=\anggrn,distance=\spr]},pattern color=PineGreen]
                    (z343) -- (z344) -- (u104) -- (u021) -- (x341) -- (x342) -- (u022) -- (u103) -- cycle;
                \draw [draw=none,pattern={Lines[angle=\anggrn,distance=\spr]},pattern color=PineGreen]
                    (y341) -- (y342) -- (u212) -- (u023) -- (x343) -- (x344) -- (u024) -- (u211) -- cycle;
                \draw [draw=none,pattern={Lines[angle=\anggrn,distance=\spr]},pattern color=PineGreen]
                    (x5) -- (x561) -- (y211) -- (y2) -- cycle;
                \draw [draw=none,pattern={Lines[angle=\anggrn,distance=\spr]},pattern color=PineGreen]
                    (x6) -- (x562) -- (y212) -- (y1) -- cycle;
                \draw [draw=none,pattern={Lines[angle=\anggrn,distance=\spr]},pattern color=PineGreen]
                    (z5) -- (z561) -- (v101) -- (v122) -- (x122) -- (x2) -- (v1) -- cycle;
                \draw [draw=none,pattern={Lines[angle=\anggrn,distance=\spr]},pattern color=PineGreen]
                    (z562) -- (z563) -- (v103) -- (v2) -- (x1) -- (x121) -- (v121) -- (v102) -- cycle;
                \draw [draw=none,pattern={Lines[angle=\anggrn,distance=\spr]},pattern color=PineGreen]
                    (z6) -- (z564) -- (v104) -- (v021) -- (w011) -- (w0) -- (v0) -- cycle;
                
                \draw [draw=none,pattern={Lines[angle=\angred,distance=\spr]},pattern color=red]
                    (y561) -- (y562) -- (z212) -- (z211) -- cycle;
                \draw [draw=none,pattern={Lines[angle=\angred,distance=\spr]},pattern color=red]
                    (x561) -- (x562) -- (y212) -- (y211) -- cycle;
                \draw [draw=none,pattern={Lines[angle=\angred,distance=\spr]},pattern color=red]
                    (y3) -- (y341) -- (u211) -- (u024) -- (x344) -- (x4) -- (u2);
                \draw [draw=none,pattern={Lines[angle=\angred,distance=\spr]},pattern color=red]
                    (y4) -- (y344) -- (u214) -- (u101) -- (z341) -- (z3) -- (u1);
                \draw [draw=none,pattern={Lines[angle=\angred,distance=\spr]},pattern color=red]
                    (x3) -- (x341) -- (u021) -- (u104) -- (z344) -- (z4) -- (u0);
                \draw [draw=none,pattern={Lines[angle=\angred,distance=\spr]},pattern color=red]
                    (x342) -- (x343) -- (u023) -- (u212) -- (y342) -- (y343) -- (u213) -- (u102) -- (z342) -- (z343) -- (u103) -- (u022);
                \draw [draw=none,pattern={Lines[angle=\angred,distance=\spr]},pattern color=red]
                    (z561) -- (z562) -- (v102) -- (v121) -- (x121) -- (x122) -- (v122) -- (v101) -- cycle;
                \draw [draw=none,pattern={Lines[angle=\angred,distance=\spr]},pattern color=red]
                    (z564) -- (z563) -- (v103) -- (v2) -- (w1) -- (w011) -- (v021) -- (v104) -- cycle;
                    
                \draw [draw=none,fill=gray!35] (x6) -- (x7) -- (y0) -- (y1) -- cycle;
                \draw [draw=none,fill=gray!35] (y6) -- (y7) -- (z0) -- (z1) -- cycle;
                \draw [draw=none,fill=gray!35] (z6) -- (z7) -- (w0) -- (v0) -- cycle;
                \draw [draw=none,fill=gray!35] (v2) -- (w1) -- (x0) -- (x1) -- cycle;
                
                \draw [thick,fill=gray!35] (x4) -- (x5) -- (y2) -- (y3) -- (u2) -- cycle;
                \draw [thick,fill=gray!35] (z2) -- (z3) -- (u1) -- (y4) -- (y5) -- cycle;
                \draw [thick,fill=gray!35] (x2) -- (x3) -- (u0) -- (z4) -- (z5) -- (v1) -- cycle;
                
                \foreach \i [evaluate={\j={\i+1}}] in {0,...,6} {
                    \draw [ultra thick] (x\i) -- (x\j) (y\i) -- (y\j) (z\i) -- (z\j);}
                
                \foreach \i [evaluate={\j={mod(\i+1,3)}}] in {0,1,2} {
                    \draw [thick,dashed] (u\i) -- (u\j) (v\i) -- (v\j);}
                    
                \draw [thick] (y1) -- (x6) (z1) -- (y6) (v2) -- (x1) (z6) -- (v0) (w0) -- (v0) (v2) -- (w1);
                
                \foreach \i in {1,...,4}{
                    \draw [thick] (x34\i) -- (u02\i) (z34\i) -- (u10\i) (y34\i) -- (u21\i) (z56\i) -- (v10\i);}
                    
                \draw [thick] (u021) -- (u104) (u101) -- (u214) (u211) -- (u024)
                    (u022) -- (u103) (u102) -- (u213) (u212) -- (u023);
                    
                \foreach \i in {1,2}{
                    \draw [thick] (x56\i) -- (y21\i) (y56\i) -- (z21\i) (x12\i) -- (v12\i);}
                    
                \draw [thick] (w011) -- (v021)
                    (v101) -- (v122) (v102) -- (v121) (v103) -- (v2) (v104) -- (v021);
                    
                \draw [ultra thick,red,shorten <=2pt,shorten >=2pt,<-]
                    ([xshift=-2pt,yshift=7pt]x3) -- ([xshift=-2pt,yshift=7pt]x341);
                    
                \draw [ultra thick,PineGreen,shorten <=2pt,shorten >=2pt,->]
                    ([xshift=-2pt,yshift=7pt]x341) -- ([xshift=-2pt,yshift=7pt]x342);
                    
                \draw [ultra thick,red,shorten <=2pt,shorten >=2pt,->]
                    ([xshift=-2pt,yshift=7pt]x342) -- ([xshift=-2pt,yshift=7pt]x343);
                    
                \draw [ultra thick,PineGreen,shorten <=2pt,shorten >=2pt,<-]
                    ([xshift=-2pt,yshift=7pt]x343) -- ([xshift=-2pt,yshift=7pt]x344);
                    
                \draw [ultra thick,red,shorten <=2pt,shorten >=2pt,<-]
                    ([xshift=-2pt,yshift=7pt]x344) -- ([xshift=-2pt,yshift=7pt]x4);
                    
                \draw [ultra thick,red,shorten <=2pt,shorten >=2pt,<-]
                    ([xshift=1pt,yshift=-7pt]u0) -- ([xshift=1pt,yshift=-7pt]u021);
                    
                \draw [ultra thick,PineGreen,shorten <=2pt,shorten >=2pt,->]
                    ([xshift=1pt,yshift=-7pt]u021) -- ([xshift=1pt,yshift=-7pt]u022);
                    
                \draw [ultra thick,red,shorten <=2pt,shorten >=2pt,->]
                    ([xshift=1pt,yshift=-7pt]u022) -- ([xshift=1pt,yshift=-7pt]u023);
                    
                \draw [ultra thick,PineGreen,shorten <=2pt,shorten >=2pt,<-]
                    ([xshift=1pt,yshift=-7pt]u023) -- ([xshift=1pt,yshift=-7pt]u024);
                    
                \draw [ultra thick,red,shorten <=2pt,shorten >=2pt,<-]
                    ([xshift=1pt,yshift=-7pt]u024) -- ([xshift=1pt,yshift=-7pt]u2);
                    
                \draw [ultra thick,red,shorten <=2pt,shorten >=2pt,<-]
                    ([xshift=-5pt,yshift=5pt]u0) -- ([xshift=-5pt,yshift=5pt]u104);
                    
                \draw [ultra thick,PineGreen,shorten <=2pt,shorten >=2pt,->]
                    ([xshift=-5pt,yshift=5pt]u104) -- ([xshift=-5pt,yshift=5pt]u103);
                    
                \draw [ultra thick,red,shorten <=2pt,shorten >=2pt,->]
                    ([xshift=-5pt,yshift=5pt]u103) -- ([xshift=-5pt,yshift=5pt]u102);
                    
                \draw [ultra thick,PineGreen,shorten <=2pt,shorten >=2pt,->]
                    ([xshift=-5pt,yshift=5pt]u102) -- ([xshift=-5pt,yshift=5pt]u101);
                    
                \draw [ultra thick,red,shorten <=2pt,shorten >=2pt,<-]
                    ([xshift=-5pt,yshift=5pt]u101) -- ([xshift=-5pt,yshift=5pt]u1);
                    
                \draw [ultra thick,red,shorten <=2pt,shorten >=2pt,<-]
                    ([xshift=5pt,yshift=-5pt]z4) -- ([xshift=5pt,yshift=-5pt]z344);
                    
                \draw [ultra thick,PineGreen,shorten <=2pt,shorten >=2pt,->]
                    ([xshift=5pt,yshift=-5pt]z344) -- ([xshift=5pt,yshift=-5pt]z343);
                    
                \draw [ultra thick,red,shorten <=2pt,shorten >=2pt,->]
                    ([xshift=5pt,yshift=-5pt]z343) -- ([xshift=5pt,yshift=-5pt]z342);
                    
                \draw [ultra thick,PineGreen,shorten <=2pt,shorten >=2pt,->]
                    ([xshift=5pt,yshift=-5pt]z342) -- ([xshift=5pt,yshift=-5pt]z341);
                    
                \draw [ultra thick,red,shorten <=2pt,shorten >=2pt,<-]
                    ([xshift=5pt,yshift=-5pt]z341) -- ([xshift=5pt,yshift=-5pt]z3);
                    
                \draw [ultra thick,red,shorten <=2pt,shorten >=2pt,->]
                    ([xshift=-5pt,yshift=-5pt]y3) -- ([xshift=-5pt,yshift=-5pt]y341);
                    
                \draw [ultra thick,PineGreen,shorten <=2pt,shorten >=2pt,->]
                    ([xshift=-5pt,yshift=-5pt]y341) -- ([xshift=-5pt,yshift=-5pt]y342);
                    
                \draw [ultra thick,red,shorten <=2pt,shorten >=2pt,<-]
                    ([xshift=-5pt,yshift=-5pt]y342) -- ([xshift=-5pt,yshift=-5pt]y343);
                    
                \draw [ultra thick,PineGreen,shorten <=2pt,shorten >=2pt,->]
                    ([xshift=-5pt,yshift=-5pt]y343) -- ([xshift=-5pt,yshift=-5pt]y344);
                    
                \draw [ultra thick,red,shorten <=2pt,shorten >=2pt,<-]
                    ([xshift=-5pt,yshift=-5pt]y344) -- ([xshift=-5pt,yshift=-5pt]y4);
                    
                \draw [ultra thick,red,shorten <=2pt,shorten >=2pt,->]
                    ([xshift=5pt,yshift=5pt]u1) -- ([xshift=5pt,yshift=5pt]u214);
                    
                \draw [ultra thick,PineGreen,shorten <=2pt,shorten >=2pt,<-]
                    ([xshift=5pt,yshift=5pt]u214) -- ([xshift=5pt,yshift=5pt]u213);
                    
                \draw [ultra thick,red,shorten <=2pt,shorten >=2pt,->]
                    ([xshift=5pt,yshift=5pt]u213) -- ([xshift=5pt,yshift=5pt]u212);
                    
                \draw [ultra thick,PineGreen,shorten <=2pt,shorten >=2pt,<-]
                    ([xshift=5pt,yshift=5pt]u212) -- ([xshift=5pt,yshift=5pt]u211);
                    
                \draw [ultra thick,red,shorten <=2pt,shorten >=2pt,<-]
                    ([xshift=5pt,yshift=5pt]u211) -- ([xshift=5pt,yshift=5pt]u2);
                    
                \draw [ultra thick,PineGreen,shorten <=2pt,shorten >=2pt,<-]
                    ([xshift=-5pt,yshift=5pt]x1) -- ([xshift=-5pt,yshift=5pt]x121);
                    
                \draw [ultra thick,red,shorten <=2pt,shorten >=2pt,->]
                    ([xshift=-5pt,yshift=5pt]x121) -- ([xshift=-5pt,yshift=5pt]x122);
                    
                \draw [ultra thick,PineGreen,shorten <=2pt,shorten >=2pt,<-]
                    ([xshift=-5pt,yshift=5pt]x122) -- ([xshift=-5pt,yshift=5pt]x2);
                    
                \draw [ultra thick,PineGreen,shorten <=2pt,shorten >=2pt,->]
                    ([xshift=5pt,yshift=-5pt]v1) -- ([xshift=5pt,yshift=-5pt]v122);
                    
                \draw [ultra thick,red,shorten <=2pt,shorten >=2pt,->]
                    ([xshift=5pt,yshift=-5pt]v121) -- ([xshift=5pt,yshift=-5pt]v122);
                    
                \draw [ultra thick,PineGreen,shorten <=2pt,shorten >=2pt,->]
                    ([xshift=5pt,yshift=-5pt]v121) -- ([xshift=5pt,yshift=-5pt]v2);
                    
                \draw [ultra thick,red,shorten <=2pt,shorten >=2pt,<-]
                    ([xshift=-5pt,yshift=-5pt]v2) -- ([xshift=-5pt,yshift=-5pt]v021);
                    
                \draw [ultra thick,PineGreen,shorten <=2pt,shorten >=2pt,<-]
                    ([xshift=-5pt,yshift=-5pt]v021) -- ([xshift=-5pt,yshift=-5pt]v0);
                    
                \draw [ultra thick,PineGreen,shorten <=1pt,shorten >=1pt,->]
                    ([yshift=-8pt]z6) -- ([yshift=-8pt]z564);
                    
                \draw [ultra thick,red,shorten <=1pt,shorten >=1pt,->]
                    ([yshift=-8pt]z564) -- ([yshift=-8pt]z563);
                    
                \draw [ultra thick,PineGreen,shorten <=1pt,shorten >=1pt,<-]
                    ([yshift=-8pt]z563) -- ([yshift=-8pt]z562);
                    
                \draw [ultra thick,red,shorten <=1pt,shorten >=1pt,->]
                    ([yshift=-8pt]z562) -- ([yshift=-8pt]z561);
                    
                \draw [ultra thick,PineGreen,shorten <=1pt,shorten >=1pt,<-]
                    ([yshift=-8pt]z561) -- ([yshift=-8pt]z5);
                    
                \draw [ultra thick,PineGreen,shorten <=1pt,shorten >=1pt,->]
                    ([yshift=8pt]v0) -- ([yshift=8pt]v104);
                    
                \draw [ultra thick,red,shorten <=1pt,shorten >=1pt,->]
                    ([yshift=8pt]v104) -- ([yshift=8pt]v103);
                    
                \draw [ultra thick,PineGreen,shorten <=1pt,shorten >=1pt,<-]
                    ([yshift=8pt]v103) -- ([yshift=8pt]v102);
                    
                \draw [ultra thick,red,shorten <=1pt,shorten >=1pt,->]
                    ([yshift=8pt]v102) -- ([yshift=8pt]v101);
                    
                \draw [ultra thick,PineGreen,shorten <=1pt,shorten >=1pt,<-]
                    ([yshift=8pt]v101) -- ([yshift=8pt]v1);
                    
                \draw [ultra thick,PineGreen,shorten <=1pt,shorten >=1pt,<-]
                    ([yshift=8pt]x5) -- ([yshift=8pt]x561);
                    
                \draw [ultra thick,red,shorten <=1pt,shorten >=1pt,->]
                    ([yshift=8pt]x561) -- ([yshift=8pt]x562);
                    
                \draw [ultra thick,PineGreen,shorten <=1pt,shorten >=1pt,->]
                    ([yshift=8pt]x562) -- ([yshift=8pt]x6);
                    
                \draw [ultra thick,PineGreen,shorten <=1pt,shorten >=1pt,<-]
                    ([yshift=-8pt]y2) -- ([yshift=-8pt]y211);
                    
                \draw [ultra thick,red,shorten <=1pt,shorten >=1pt,->]
                    ([yshift=-8pt]y211) -- ([yshift=-8pt]y212);
                    
                \draw [ultra thick,PineGreen,shorten <=1pt,shorten >=1pt,->]
                    ([yshift=-8pt]y212) -- ([yshift=-8pt]y1);
                    
                \draw [ultra thick,PineGreen,shorten <=1pt,shorten >=1pt,<-]
                    ([xshift=-8pt]y5) -- ([xshift=-8pt]y561);
                    
                \draw [ultra thick,red,shorten <=1pt,shorten >=1pt,->]
                    ([xshift=-8pt]y561) -- ([xshift=-8pt]y562);
                    
                \draw [ultra thick,PineGreen,shorten <=1pt,shorten >=1pt,<-]
                    ([xshift=-8pt]y562) -- ([xshift=-8pt]y6);
                    
                \draw [ultra thick,PineGreen,shorten <=1pt,shorten >=1pt,<-]
                    ([xshift=8pt]z2) -- ([xshift=8pt]z211);
                    
                \draw [ultra thick,red,shorten <=1pt,shorten >=1pt,->]
                    ([xshift=8pt]z211) -- ([xshift=8pt]z212);
                    
                \draw [ultra thick,PineGreen,shorten <=1pt,shorten >=1pt,<-]
                    ([xshift=8pt]z212) -- ([xshift=8pt]z1);
                        
                \begin{scope}[decoration={
                    markings,
                    mark=at position 0 with {\arrow[scale=1.5,>=stealth,xshift=2.5\pgflinewidth]{<}}}
                    ]
                    \draw [thick,postaction=decorate] (barycentric cs:u0=1,u1=1,u2=1) to [bend right=40] (5,1)
                        node [right] {an $a$-hexagon};
                    
                    \draw [thick,postaction=decorate] (barycentric cs:x1=1,x2=1,v1=1,v2=1) to [bend right=30] (0,-4)
                        node [right] {an $a$-stripe};
                    
                    \draw [thick,postaction=decorate] (barycentric cs:z5=1,z6=1,v1=1,v0=1) to [bend left=45] (-8,10)
                        node [right] {a $b$-stripe};
                        
                    \draw [thick,postaction=decorate] (barycentric cs:x2=1,x3=1,u0=1,z4=1,z5=1,v1=1) to [bend left=15] (3.5,-1)
                        node [below] {a vertex disc};
                \end{scope}
            \end{tikzpicture}
            \caption{Part of a stripe pattern.}\label{fig:zebra}
        \end{figure}
        
    \subsection{Boundary labelling}\label{subsec:bnd-lbl}
        
        We now turn our attention to the boundary of $\Sigma$.
        The (oriented) boundary edges of $\Sigma$ are either part of a vertex disc, or they bound a stripe and are then labelled by some letter in $\left\{a^{\pm1},b^{\pm1}\right\}$.
        Moreover, the concatenation of the oriented labels of the boundary edges along a $1$-handle $\mathcal{H}$ of $\Sigma$ mapping via $f$ to $g_\mathcal{H}$ is the (reduced, but possibly not cyclically reduced) word $\Phi\left(g_\mathcal{H}\right)$ (see Figure \ref{fig:new-handle}).
        Now pick a boundary component $\partial_j$ of $\Sigma$.
        The image of $\partial_j$ under $f$ represents a word $g_1\cdots g_\ell=g^k$ for some $k\in\mathbb{N}_{\geq1}$, where each $g_i$ is an element in $\pi_1X$ represented by a single edge loop in $X$, and $g$ is the element of $\pi_1X$ with respect to which $f:\Sigma\rightarrow X$ is an admissible surface.
        It follows that the labelling of $\partial_j$ obtained by concatenating the labels of the $a$- and $b$-stripes along it is a (not necessarily reduced) word representing $\Phi\left(g_1\right)\cdots\Phi\left(g_\ell\right)$.
        
        We want to modify $\Sigma$ so that the labelling of $\partial_j$ becomes $\Phi\left(g_1\cdots g_\ell\right)=\Phi\left(g\right)^k$.
        We will use the letter-quasimorphism condition to do so.
        Consider two successive sections of $\partial_j$ that are labelled by $\Phi\left(g_i\right)$ and $\Phi\left(g_{i+1}\right)$.
        Definition \ref{def:let-qm} says that either $\Phi\left(g_i\right)\Phi\left(g_{i+1}\right)=\Phi\left(g_ig_{i+1}\right)$ (in the degenerate case), or there are alternating words $u,v,w\in F_2$, and letters $x_1,x_2,x_3$ in $\left\{a^{\pm1}\right\}$ or $\left\{b^{\pm1}\right\}$ such that
        \[
            \Phi\left(g_i\right)=u^{-1}x_1v,\quad\Phi\left(g_{i+1}\right)=v^{-1}x_2w,\quad\Phi\left(g_ig_{i+1}\right)^{-1}=w^{-1}x_3u
        \]
        (as reduced words).
        We can use this data to glue some new stripes and (in the non-degenerate case) one $a$- or $b$-hexagon to $\partial_j$ as in Figure \ref{fig:cor-bnd-lab}, modifying $\Sigma$ by a homeomorphism so that the portion of $\partial_j$ that was labelled by $\Phi\left(g_i\right)\Phi\left(g_{i+1}\right)$ is now labelled by $\Phi\left(g_ig_{i+1}\right)$.
        Note that this operation relies on the fact that the labelling of each $1$-handle on $\partial\Sigma$ is a reduced word, and this property still hold after gluing the new pieces.
        
        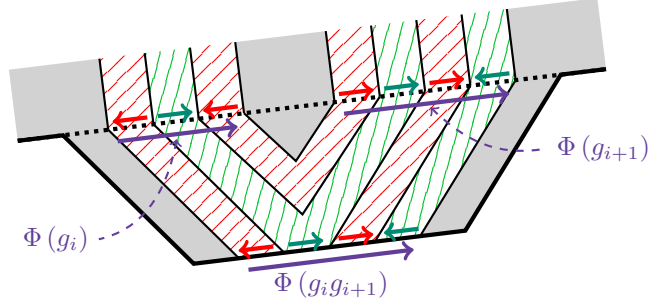
\begin{figure}[hbt]
            \centering
            \begin{tikzpicture}[every node/.style={rectangle,draw=none,fill=none},rotate=7,xscale=.6,yscale=.95]
                \foreach \i in {0,...,3}{
                    \coordinate (a\i) at (\i,0);
                    \coordinate (ap\i) at (\i,1);}
                \foreach \i in {0,...,4}{
                    \coordinate (b\i) at (\i+5,0);
                    \coordinate (bp\i) at (\i+5,1);
                    \coordinate (c\i) at (2.5+\i,-2);}
                \coordinate (d1) at (4,-.75);
                \coordinate (d2) at (4,-1.5);
                    
                \draw [draw=none,pattern={Lines[angle=\angred,distance=\spr]},pattern color=red]
                    (ap0) -- (a0) -- (c0) -- (c1) -- (a1) -- (ap1) -- cycle;
                \draw [draw=none,pattern={Lines[angle=\anggrn,distance=\spr]},pattern color=PineGreen]
                    (ap1) -- (a1) -- (c1) -- (c2) -- (b2) -- (bp2) -- (bp1) -- (b1) -- (d2) -- (a2) -- (ap2) -- cycle;
                \draw [draw=none,pattern={Lines[angle=\angred,distance=\spr]},pattern color=red]
                    (ap2) -- (a2) -- (d2) -- (b1) -- (bp1) -- (bp0) -- (b0) -- (d1) -- (a3) -- (ap3) -- cycle;
                \draw [draw=none,pattern={Lines[angle=\angred,distance=\spr]},pattern color=red]
                    (bp2) -- (b2) -- (c2) -- (c3) -- (b3) -- (bp3) -- cycle;
                \draw [draw=none,pattern={Lines[angle=\anggrn,distance=\spr]},pattern color=PineGreen]
                    (bp3) -- (b3) -- (c3) -- (c4) -- (b4) -- (bp4) -- cycle;
                
                \draw [draw=none,fill=gray!35] (-2,1) -- (ap0) -- (a0) -- (c0) -- (1.5,-2) -- (-1,0) -- (-2,0) -- cycle;
                \draw [draw=none,fill=gray!35] (11,1) -- (bp4) -- (b4) -- (c4) -- (7.5,-2) -- (10,0) -- (11,0) -- cycle;
                \draw [draw=none,fill=gray!35] (ap3) -- (a3) -- (d1) -- (b0) -- (bp0) -- cycle;
                
                \foreach \i in {0,...,3}{
                    \draw [thick] (ap\i) -- (a\i);}
                \foreach \i in {0,...,4}{
                    \draw [thick] (bp\i) -- (b\i);}
                \draw [thick] (a0) -- (c0) (a1) -- (c1) (a2) -- (d2) -- (b1) (a3) -- (d1) -- (b0)
                    (b2) -- (c2) (b3) -- (c3) (b4) -- (c4);
                    
                \draw [ultra thick,dotted] (-2,0) -- (11,0);
                \draw [ultra thick] (-2,0) -- (-1,0) -- (1.5,-2) -- (7.5,-2) -- (10,0) -- (11,0);
                
                \foreach \i [evaluate={\j={int(\i+1)}}] in {0,2}{
                    \draw [ultra thick,red,shorten <=2pt,shorten >=2pt,<-]
                        ([yshift=+3pt]a\i) -- ([yshift=+3pt]a\j);}
                \draw [ultra thick,PineGreen,shorten <=2pt,shorten >=2pt,->]
                    ([yshift=+3pt]b1) -- ([yshift=+3pt]b2);
                \draw [ultra thick,PineGreen,shorten <=2pt,shorten >=2pt,<-]
                    ([yshift=+3pt]b3) -- ([yshift=+3pt]b4);
                \draw [ultra thick,PineGreen,shorten <=2pt,shorten >=2pt,->]
                    ([yshift=+3pt]a1) -- ([yshift=+3pt]a2);
                \draw [ultra thick,red,shorten <=2pt,shorten >=2pt,<-]
                    ([xshift=-3pt,yshift=+3pt]c0) -- ([xshift=-3pt,yshift=+3pt]c1);
                \draw [ultra thick,PineGreen,shorten <=2pt,shorten >=2pt,->]
                    ([yshift=+3pt]c1) -- ([yshift=+3pt]c2);
                \draw [ultra thick,red,shorten <=2pt,shorten >=2pt,->]
                    ([xshift=3pt,yshift=+3pt]c2) -- ([xshift=3pt,yshift=+3pt]c3);
                \draw [ultra thick,PineGreen,shorten <=2pt,shorten >=2pt,<-]
                    ([xshift=3pt,yshift=+3pt]c3) -- ([xshift=3pt,yshift=+3pt]c4);
                \draw [ultra thick,red,shorten <=2pt,shorten >=2pt,->]
                    ([yshift=+3pt]b0) -- ([yshift=+3pt]b1);
                \draw [ultra thick,red,shorten <=2pt,shorten >=2pt,->]
                    ([yshift=+3pt]b2) -- ([yshift=+3pt]b3);
                    
                \draw [ultra thick,shorten <=3pt,shorten >=3pt,->,RoyalPurple] ([yshift=-5pt]a0) -- ([yshift=-5pt]a3) 
                    coordinate [midway,below] (lab1);
                \draw [<-,RoyalPurple,dashed,thick,shorten <=2pt] (lab1) to [bend left=30] (-.5,-1.5) node [left] {$\Phi\left(g_{i}\right)$};
                \draw [ultra thick,shorten <=3pt,shorten >=3pt,->,RoyalPurple] ([yshift=-5pt]b0) -- ([yshift=-5pt]b4)
                    coordinate [midway,below] (lab2);
                \draw [<-,RoyalPurple,dashed,thick,shorten <=2pt] (lab2) to [bend right=20] (9.5,-1) node [right] {$\Phi\left(g_{i+1}\right)$};
                \draw [ultra thick,shorten <=3pt,shorten >=3pt,->,RoyalPurple] ([yshift=-5pt]c0) -- ([yshift=-5pt]c4)
                    node [midway,below] {$\Phi\left(g_ig_{i+1}\right)$};
            \end{tikzpicture}
            \caption{Correcting the boundary labelling. (The former boundary is depicted as a dotted line and the new one as a full line.)}\label{fig:cor-bnd-lab}
        \end{figure}
        
        By repeating this process a finite number of times (as most as many times as the number of $1$-handles of $\Sigma$), we obtain a surface homeomorphic to $\Sigma$ with a new stripe pattern.
        Each boundary component is now labelled by a word representing some positive power of $\Phi(g)$ --- and which is therefore cyclically reduced.
        Notice for instance in Figure \ref{fig:cor-bnd-lab} that this construction has eliminated a pair of cancelling edges.
        
        Since boundary components are labelled by powers of $\Phi(g)=a_1b_1\cdots a_\ell b_\ell$, with $a_i\in\left\{a^{\pm1}\right\}$ and $b_i\in\left\{b^{\pm1}\right\}$, each boundary edge bounding an $a$-stripe corresponds to some $a_i$ and each boundary edge bounding a $b$-stripe corresponds to some $b_i$.
        We will remember the index $i$ as part of the stripe pattern.
        Hence, reading the successive indices of the edges along a boundary component of $\Sigma$ yields a cyclic permutation of a positive iterate of the sequence $\left(1,1,\dots,\ell,\ell\right)$.
        
        \begin{ex}
            Suppose that $X$ is a bouquet of two oriented circles labelled by $a$ and $b$, so that $\pi_1X=F(a,b)$.
            Let $g=[a,b]=aba^{-1}b^{-1}\in\pi_1X$.
            There is a letter-quasimorphism $\Phi:\pi_1X\rightarrow\mathcal{A}$ given by
            \[
                a^{m_1}b^{n_1}\cdots a^{m_k}b^{n_k}\longmapsto a^{\sign\left(m_1\right)}b^{\sign\left(n_1\right)}\cdots a^{\sign\left(m_k\right)}b^{\sign\left(n_k\right)},
            \]
            with $m_i,n_i\in\mathbb{Z}$, all non-zero except possibly for $m_1$ and $n_k$ (see \cite{heuer-raags}*{Example 4.2}).
            
            Applying the above construction to the admissible surface $f:\left(\Sigma,\partial\Sigma\right)\rightarrow\left(X,g\right)$ which is a once-punctured torus with generators mapping to $X$ in the standard way yields the stripe pattern of Figure \ref{fig:ex:stripeptrn}.
            The only boundary component of $\Sigma$ is labelled by $\left(1,1,2,2\right)$.

            \begin{figure}[htb]
                \centering
                \begin{tikzpicture}[every node/.style={rectangle,draw=none,fill=none},rotate=45,scale=.8]
                    \foreach \i in {0,...,3}{
                        \coordinate (x\i) at (\i*360/4: 3cm);
                        \coordinate (y\i) at (\i*360/4: .75cm);}
                    
                    \draw [draw=none,pattern={Lines[angle=\anggrn,distance=\spr]},pattern color=PineGreen]
                        (x0) -- (x1) -- (y1) -- (y0) -- cycle;
                    
                    \draw [draw=none,pattern={Lines[angle=\anggrn,distance=\spr]},pattern color=PineGreen]
                        (x2) -- (x3) -- (y3) -- (y2) -- cycle;
                    
                    \draw [draw=none,pattern={Lines[angle=\angred,distance=\spr]},pattern color=red]
                        (x1) -- (x2) -- (y2) -- (y1) -- cycle;
                    
                    \draw [draw=none,pattern={Lines[angle=\angred,distance=\spr]},pattern color=red]
                        (x0) -- (x3) -- (y3) -- (y0) -- cycle;
                    
                    \foreach \i in {0,...,3}{\draw [thick] (x\i) -- (y\i);}
                    
                    \foreach \i [evaluate={\j=int(mod(\i+1,4))}] in {0,...,3}{
                        \draw [ultra thick] (y\i) -- (y\j);}
                        
                    \begin{scope}[decoration={
                        markings,
                        mark=at position .5 with {\arrow[scale=1.5]{stealth}}}
                        ] 
                        \draw [thick,dashed,postaction={decorate}] (x1) -- (x0);
                        \draw [thick,dashed,postaction={decorate}] (x2) -- (x3);
                    \end{scope}
                    
                    \begin{scope}[decoration={
                        markings,
                        mark=at position .48 with {\arrow[scale=1.5]{stealth}},
                        mark=at position .52 with {\arrow[scale=1.5]{stealth}}}
                        ] 
                        \draw [thick,dashed,postaction={decorate}] (x2) -- (x1);
                        \draw [thick,dashed,postaction={decorate}] (x3) -- (x0);
                    \end{scope}
                            
                    \draw [ultra thick,PineGreen,->] ([xshift=4pt,yshift=4pt]y1) -- ([xshift=4pt,yshift=4pt]y0)
                        node [midway,above,circle,fill=white,draw=none,inner sep=.1pt] {$1$};
                    \draw [ultra thick,PineGreen,->] ([xshift=-4pt,yshift=-4pt]y2) -- ([xshift=-4pt,yshift=-4pt]y3)
                        node [midway,below,circle,fill=white,draw=none,inner sep=.1pt] {$2$};
                    \draw [ultra thick,red,->] ([xshift=4pt,yshift=-4pt]y3) -- ([xshift=4pt,yshift=-4pt]y0)
                        node [midway,right,circle,fill=white,draw=none,inner sep=.1pt] {$2$};
                    \draw [ultra thick,red,->] ([xshift=-4pt,yshift=4pt]y2) -- ([xshift=-4pt,yshift=4pt]y1)
                        node [midway,left,circle,fill=white,draw=none,inner sep=.1pt] {$1$};
                            
                    \draw [ultra thick,PineGreen,->,shorten <=10pt,shorten >=10pt] ([xshift=-4pt,yshift=-4pt]x1) -- ([xshift=-4pt,yshift=-4pt]x0);
                    \draw [ultra thick,PineGreen,->,shorten <=10pt,shorten >=10pt] ([xshift=4pt,yshift=4pt]x2) -- ([xshift=4pt,yshift=4pt]x3);
                    \draw [ultra thick,red,->,shorten <=10pt,shorten >=10pt] ([xshift=-4pt,yshift=4pt]x3) -- ([xshift=-4pt,yshift=4pt]x0);
                    \draw [ultra thick,red,->,shorten <=10pt,shorten >=10pt] ([xshift=4pt,yshift=-4pt]x2) -- ([xshift=4pt,yshift=-4pt]x1);
                    
                    \draw node at (-.5,1.4) {\Huge$\circlearrowleft$};
                \end{tikzpicture}
                \caption{Stripe pattern on an admissible surface for $g=[a,b]\in F(a,b)$ (black arrows indicate edge identifications).}\label{fig:ex:stripeptrn}
            \end{figure}
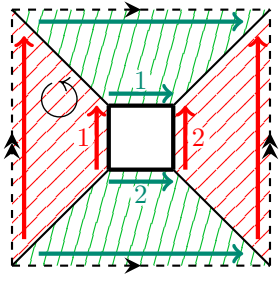
        \end{ex}
        
    \subsection{Unzipping \texorpdfstring{$\Sigma$}{Sigma}}
    
        With its current stripe pattern, $\Sigma$ is subdivided into:
        \begin{itemize}
            \item The union $\Sigma^D$ of its closed vertex discs,
            \item The union $\Sigma^a$ of its closed $a$-stripes and $a$-hexagons, and
            \item The union $\Sigma^b$ of its closed $b$-stripes and $b$-hexagons.
        \end{itemize}
        Consider the boundary of those different regions:
        \[
            \Gamma\coloneqq\left(\Sigma^a\cap\Sigma^b\right)\cup\left(\Sigma^a\cap\Sigma^D\right)\cup\left(\Sigma^b\cap\Sigma^D\right).
        \]
        By observing the local structure of the stripe pattern, one can see that $\Gamma$ is an embedded graph in $\Sigma$.
        
        We want to perform one last modification on the stripe pattern to remove all singular points of $\Gamma$ (i.e. to ensure that $\Gamma$ is a $1$-dimensional submanifold of $\Sigma$).
        Note that the only points where $\Gamma$ might not be locally homeomorphic to a line are the apices of degenerate cellular discs (see Figure \ref{fig:new-discs-dg}).
        We can \emph{unzip} each degenerate cellular disc as indicated in Figure \ref{fig:unzip}, extending an existing vertex disc between the $a$- and $b$-stripes meeting at the apex of the triangle.
        
        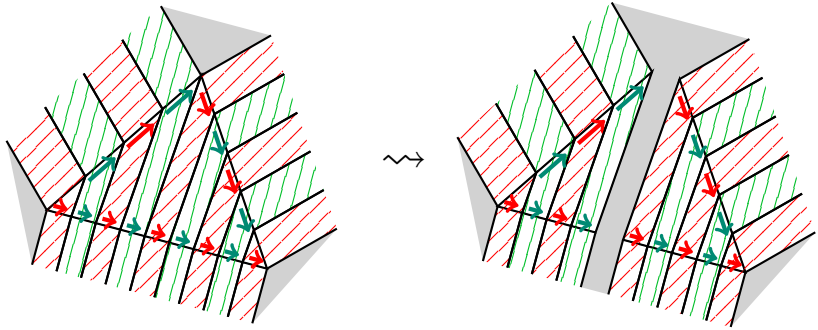
\begin{figure}[hbt]
            \centering
            \begin{subfigure}[c]{15em}
                \centering
                \begin{tikzpicture}[every node/.style={rectangle,draw=none,fill=none},rotate=-15,scale=.75]
                    \coordinate (a) at (0,3);
                    \coordinate (b) at (-2,0);
                    \coordinate (c) at (2,0);
                    
                    \foreach \i in {0,...,9}{
                        \coordinate (bc\i) at ({-2+4*\i/9},0);
                        \coordinate (bcz\i) at ({-2+4*\i/9},-1);}
                    \foreach \i in {0,...,4}{
                        \coordinate (ba\i) at ({-2+2*\i/4},{3*\i/4});
                        \coordinate (baz\i) at ({-2+2*\i/4-1},{3*\i/4+1});}
                    \foreach \i in {0,...,5}{
                        \coordinate (ac\i) at ({2*\i/5},{3-3*\i/5});
                        \coordinate (acz\i) at ({2*\i/5+1},{3-3*\i/5+1});}
                    
                    \draw [draw=none,fill=gray!35]
                            (a) -- (baz4) -- (acz0);
                    \draw [draw=none,fill=gray!35]
                            (b) -- (baz0) -- (bcz0);
                    \draw [draw=none,fill=gray!35]
                            (c) -- (acz5) -- (bcz9);
                    
                    \foreach \i [evaluate={\j={int(\i-1)}}] in {1,3}{
                        \draw [pattern={Lines[angle=\angred,distance=\spr]},pattern color=red]
                            (bc\i) -- (ba\i) -- (ba\j) -- (bc\j) -- cycle;
                        \draw [thick]
                            (bc\i) -- (ba\i) -- (ba\j) -- (bc\j) -- cycle;
                        \draw [draw=none,pattern={Lines[angle=\angred,distance=\spr]},pattern color=red]
                            (ba\i) -- (baz\i) -- (baz\j) -- (ba\j) -- cycle;
                        \draw [draw=none,pattern={Lines[angle=\angred,distance=\spr]},pattern color=red]
                            (bc\i) -- (bcz\i) -- (bcz\j) -- (bc\j) -- cycle;
                        \draw [thick] (ba\i) -- (baz\i) (ba\j) -- (baz\j);
                        \draw [thick] (bc\i) -- (bcz\i) (bc\j) -- (bcz\j);
                        \draw [ultra thick,red,shorten <=2pt,shorten >=2pt,->]
                            ([yshift=+3pt]bc\j) -- ([yshift=+3pt]bc\i);
                        \ifnum\j>1
                            \draw [ultra thick,red,shorten <=3pt,shorten >=3pt,->]
                                ([yshift=-5pt]ba\j) -- ([yshift=-5pt]ba\i);
                        \fi}
                            
                    \foreach \i [evaluate={\j={int(\i-1)}}] in {2,4}{
                        \draw [pattern={Lines[angle=\anggrn,distance=\spr]},pattern color=PineGreen]
                            (bc\i) -- (ba\i) -- (ba\j) -- (bc\j) -- cycle;
                        \draw [thick]
                            (bc\i) -- (ba\i) -- (ba\j) -- (bc\j) -- cycle;
                        \draw [draw=none,pattern={Lines[angle=\anggrn,distance=\spr]},pattern color=PineGreen]
                            (ba\i) -- (baz\i) -- (baz\j) -- (ba\j) -- cycle;
                        \draw [draw=none,pattern={Lines[angle=\anggrn,distance=\spr]},pattern color=PineGreen]
                            (bc\i) -- (bcz\i) -- (bcz\j) -- (bc\j) -- cycle;
                        \draw [thick] (ba\i) -- (baz\i) (ba\j) -- (baz\j);
                        \draw [thick] (bc\i) -- (bcz\i) (bc\j) -- (bcz\j);
                        \draw [ultra thick,PineGreen,shorten <=2pt,shorten >=2pt,->]
                            ([yshift=+3pt]bc\j) -- ([yshift=+3pt]bc\i);
                        \draw [ultra thick,PineGreen,shorten <=3pt,shorten >=3pt,->]
                            ([yshift=-5pt]ba\j) -- ([yshift=-5pt]ba\i);}
                            
                    \foreach \i [evaluate={\j={int(\i-1)}},evaluate={\ip={int(\i+4)}},evaluate={\jp={int(\j+4)}}] in {1,3,5}{
                        \draw [pattern={Lines[angle=\angred,distance=\spr]},pattern color=red]
                            (bc\ip) -- (ac\i) -- (ac\j) -- (bc\jp) -- cycle;
                        \draw [thick]
                            (bc\ip) -- (ac\i) -- (ac\j) -- (bc\jp) -- cycle;
                        \draw [draw=none,pattern={Lines[angle=\angred,distance=\spr]},pattern color=red]
                            (ac\i) -- (acz\i) -- (acz\j) -- (ac\j) -- cycle;
                        \draw [draw=none,pattern={Lines[angle=\angred,distance=\spr]},pattern color=red]
                            (bc\ip) -- (bcz\ip) -- (bcz\jp) -- (bc\jp) -- cycle;
                        \draw [thick] (ac\i) -- (acz\i) (ac\j) -- (acz\j);
                        \draw [thick] (bc\ip) -- (bcz\ip) (bc\jp) -- (bcz\jp);
                        \draw [ultra thick,red,shorten <=2pt,shorten >=2pt,->]
                            ([yshift=+3pt]bc\jp) -- ([yshift=+3pt]bc\ip);
                        \ifnum\i<5
                            \draw [ultra thick,red,shorten <=3pt,shorten >=3pt,->]
                                ([yshift=-5pt]ac\j) -- ([yshift=-5pt]ac\i);
                        \fi}
                        
                    \foreach \i [evaluate={\j={int(\i-1)}},evaluate={\ip={int(\i+4)}},evaluate={\jp={int(\j+4)}}] in {2,4}{
                        \draw [pattern={Lines[angle=\anggrn,distance=\spr]},pattern color=PineGreen]
                            (bc\ip) -- (ac\i) -- (ac\j) -- (bc\jp) -- cycle;
                        \draw [thick]
                            (bc\ip) -- (ac\i) -- (ac\j) -- (bc\jp) -- cycle;
                        \draw [draw=none,pattern={Lines[angle=\anggrn,distance=\spr]},pattern color=PineGreen]
                            (ac\i) -- (acz\i) -- (acz\j) -- (ac\j) -- cycle;
                        \draw [draw=none,pattern={Lines[angle=\anggrn,distance=\spr]},pattern color=PineGreen]
                            (bc\ip) -- (bcz\ip) -- (bcz\jp) -- (bc\jp) -- cycle;
                        \draw [thick] (ac\i) -- (acz\i) (ac\j) -- (acz\j);
                        \draw [thick] (bc\ip) -- (bcz\ip) (bc\jp) -- (bcz\jp);
                        \draw [ultra thick,PineGreen,shorten <=2pt,shorten >=2pt,->]
                            ([yshift=+3pt]bc\jp) -- ([yshift=+3pt]bc\ip);
                        \draw [ultra thick,PineGreen,shorten <=3pt,shorten >=3pt,->]
                            ([yshift=-5pt]ac\j) -- ([yshift=-5pt]ac\i);}
                \end{tikzpicture}
            \end{subfigure}
            $\mathlarger{\mathlarger{\mathlarger{\mathlarger{\rightsquigarrow}}}}$
            \begin{subfigure}[c]{15em}
                \centering
                \begin{tikzpicture}[every node/.style={rectangle,draw=none,fill=none},rotate=-15,scale=.75]
                    \coordinate (aL) at (-.25,3);
                    \coordinate (aR) at (.25,3);
                    \coordinate (b) at (-2.25,0);
                    \coordinate (c) at (2.25,0);
                    
                    \foreach \i in {0,...,4}{
                        \coordinate (bcL\i) at ({-2.25+4*\i/9},0);
                        \coordinate (bczL\i) at ({-2.25+4*\i/9},-1);}
                    \foreach \i in {4,...,9}{
                        \coordinate (bcR\i) at ({-1.75+4*\i/9},0);
                        \coordinate (bczR\i) at ({-1.75+4*\i/9},-1);}
                    \foreach \i in {0,...,4}{
                        \coordinate (ba\i) at ({-2.25+2*\i/4},{3*\i/4});
                        \coordinate (baz\i) at ({-2.25+2*\i/4-1},{3*\i/4+1});}
                    \foreach \i in {0,...,5}{
                        \coordinate (ac\i) at ({2*\i/5+.25},{3-3*\i/5});
                        \coordinate (acz\i) at ({2*\i/5+1.25},{3-3*\i/5+1});}
                    
                    \draw [draw=none,fill=gray!35]
                            (b) -- (baz0) -- (bczL0);
                    \draw [draw=none,fill=gray!35]
                            (c) -- (acz5) -- (bczR9);
                    \draw [draw=none,fill=gray!35]
                            (baz4) -- (ba4) -- (bcL4) -- (bczL4) -- (bczR4) -- (bcR4) -- (ac0) -- (acz0);
                    
                    \foreach \i [evaluate={\j={int(\i-1)}}] in {1,3}{
                        \draw [pattern={Lines[angle=\angred,distance=\spr]},pattern color=red]
                            (bcL\i) -- (ba\i) -- (ba\j) -- (bcL\j) -- cycle;
                        \draw [thick]
                            (bcL\i) -- (ba\i) -- (ba\j) -- (bcL\j) -- cycle;
                        \draw [draw=none,pattern={Lines[angle=\angred,distance=\spr]},pattern color=red]
                            (ba\i) -- (baz\i) -- (baz\j) -- (ba\j) -- cycle;
                        \draw [draw=none,pattern={Lines[angle=\angred,distance=\spr]},pattern color=red]
                            (bcL\i) -- (bczL\i) -- (bczL\j) -- (bcL\j) -- cycle;
                        \draw [thick] (ba\i) -- (baz\i) (ba\j) -- (baz\j);
                        \draw [thick] (bcL\i) -- (bczL\i) (bcL\j) -- (bczL\j);
                        \draw [ultra thick,red,shorten <=2pt,shorten >=2pt,->]
                            ([yshift=+3pt]bcL\j) -- ([yshift=+3pt]bcL\i);
                        \ifnum\j>1
                            \draw [ultra thick,red,shorten <=3pt,shorten >=3pt,->]
                                ([yshift=-5pt]ba\j) -- ([yshift=-5pt]ba\i);
                        \fi}
                            
                    \foreach \i [evaluate={\j={int(\i-1)}}] in {2,4}{
                        \draw [pattern={Lines[angle=\anggrn,distance=\spr]},pattern color=PineGreen]
                            (bcL\i) -- (ba\i) -- (ba\j) -- (bcL\j) -- cycle;
                        \draw [thick]
                            (bcL\i) -- (ba\i) -- (ba\j) -- (bcL\j) -- cycle;
                        \draw [draw=none,pattern={Lines[angle=\anggrn,distance=\spr]},pattern color=PineGreen]
                            (ba\i) -- (baz\i) -- (baz\j) -- (ba\j) -- cycle;
                        \draw [draw=none,pattern={Lines[angle=\anggrn,distance=\spr]},pattern color=PineGreen]
                            (bcL\i) -- (bczL\i) -- (bczL\j) -- (bcL\j) -- cycle;
                        \draw [thick] (ba\i) -- (baz\i) (ba\j) -- (baz\j);
                        \draw [thick] (bcL\i) -- (bczL\i) (bcL\j) -- (bczL\j);
                        \draw [ultra thick,PineGreen,shorten <=2pt,shorten >=2pt,->]
                            ([yshift=+3pt]bcL\j) -- ([yshift=+3pt]bcL\i);
                        \draw [ultra thick,PineGreen,shorten <=3pt,shorten >=3pt,->]
                            ([yshift=-5pt]ba\j) -- ([yshift=-5pt]ba\i);}
                            
                    \foreach \i [evaluate={\j={int(\i-1)}},evaluate={\ip={int(\i+4)}},evaluate={\jp={int(\j+4)}}] in {1,3,5}{
                        \draw [pattern={Lines[angle=\angred,distance=\spr]},pattern color=red]
                            (bcR\ip) -- (ac\i) -- (ac\j) -- (bcR\jp) -- cycle;
                        \draw [thick]
                            (bcR\ip) -- (ac\i) -- (ac\j) -- (bcR\jp) -- cycle;
                        \draw [draw=none,pattern={Lines[angle=\angred,distance=\spr]},pattern color=red]
                            (ac\i) -- (acz\i) -- (acz\j) -- (ac\j) -- cycle;
                        \draw [draw=none,pattern={Lines[angle=\angred,distance=\spr]},pattern color=red]
                            (bcR\ip) -- (bczR\ip) -- (bczR\jp) -- (bcR\jp) -- cycle;
                        \draw [thick] (ac\i) -- (acz\i) (ac\j) -- (acz\j);
                        \draw [thick] (bcR\ip) -- (bczR\ip) (bcR\jp) -- (bczR\jp);
                        \draw [ultra thick,red,shorten <=2pt,shorten >=2pt,->]
                            ([yshift=+3pt]bcR\jp) -- ([yshift=+3pt]bcR\ip);
                        \ifnum\i<5
                            \draw [ultra thick,red,shorten <=3pt,shorten >=3pt,->]
                                ([yshift=-5pt]ac\j) -- ([yshift=-5pt]ac\i);
                        \fi}
                        
                    \foreach \i [evaluate={\j={int(\i-1)}},evaluate={\ip={int(\i+4)}},evaluate={\jp={int(\j+4)}}] in {2,4}{
                        \draw [pattern={Lines[angle=\anggrn,distance=\spr]},pattern color=PineGreen]
                            (bcR\ip) -- (ac\i) -- (ac\j) -- (bcR\jp) -- cycle;
                        \draw [thick]
                            (bcR\ip) -- (ac\i) -- (ac\j) -- (bcR\jp) -- cycle;
                        \draw [draw=none,pattern={Lines[angle=\anggrn,distance=\spr]},pattern color=PineGreen]
                            (ac\i) -- (acz\i) -- (acz\j) -- (ac\j) -- cycle;
                        \draw [draw=none,pattern={Lines[angle=\anggrn,distance=\spr]},pattern color=PineGreen]
                            (bcR\ip) -- (bczR\ip) -- (bczR\jp) -- (bcR\jp) -- cycle;
                        \draw [thick] (ac\i) -- (acz\i) (ac\j) -- (acz\j);
                        \draw [thick] (bcR\ip) -- (bczR\ip) (bcR\jp) -- (bczR\jp);
                        \draw [ultra thick,PineGreen,shorten <=2pt,shorten >=2pt,->]
                            ([yshift=+3pt]bcR\jp) -- ([yshift=+3pt]bcR\ip);
                        \draw [ultra thick,PineGreen,shorten <=3pt,shorten >=3pt,->]
                            ([yshift=-5pt]ac\j) -- ([yshift=-5pt]ac\i);}
                \end{tikzpicture}
            \end{subfigure}
            \caption{Unzipping a degenerate cellular disc.}\label{fig:unzip}
        \end{figure}
    
        After unzipping one cellular disc, one needs to keep unzipping along the same leaf of $\Gamma$.
        One can iterate until there is no more degenerate cellular disc.
        After this process, $\Gamma$ is locally homeomorphic to a line.
        We call \emph{unzipping path} each of the paths in the original surface along which we have unzipped the cellular discs.
        Note that the unzipping process might have connected several vertex discs to one another; however, if any region resulting from these identifications contained a non-contractible simple closed curve $\beta$, then $\beta$ could be homotoped to an \emph{unzipping loop} --- i.e. an unzipping path which is a simple closed curve --- in the original surface; but then $\beta$ must be contractible by the following claim:
        \begin{claim*}
            Any unzipping loop has contractible image under $f:\Sigma\rightarrow X$ and is therefore contractible by incompressibility.
        \end{claim*}
        \begin{proof}
            Consider a cellular disc that is unzipped by the above process (this cellular disc is of the degenerate type --- see Figure \ref{fig:new-discs-dg}).
            Such a cellular disc had its boundary edges mapping to loops representing elements $g_1$, $g_2$, and $\left(g_1g_2\right)^{-1}$ of $\pi_1X$ under $f:\Sigma\rightarrow X$ (see \S{}\ref{subsec:ext-cell}).
            The restriction of any unzipping path to this cellular disc is the singular leaf (i.e. the middle edge in Figure \ref{fig:new-discs-dg}); this is homotopic in $\Sigma$ to the path following one of the boundary edges and half of the next edge, which maps to a loop representing $g_1g_1^{-1}=1$ in $\pi_1X$.
            
            Now any unzipping loop alternates between vertex discs in the original surface (which map under $f$ to the vertex of $X$), and the singular leaves of cellular discs of the degenerate type (which have contractible image under $f$ as explained above).
            This proves the claim.
        \end{proof}        
        Therefore, at the end of the unzipping process, the new vertex discs are simply connected and are therefore topological discs.
        
        Moreover, $\Gamma$ is compact as its intersection with each vertex disc, $1$-handle, or cellular disc of $\Sigma$ is compact, and $\Sigma$ has only finitely many such pieces.
        Therefore, $\Gamma$ is a compact $1$-dimensional submanifold of $\Sigma$.
        We then say that the stripe pattern is an \emph{unzipped stripe pattern}.
        
        \subsection{The angle structure}
        
        Before constructing the angle structure on $\Sigma$, we describe the cellulation that will support it.
        
        An \emph{$a$-region} of $\Sigma$ is a connected component of $\Sigma^a$, while a \emph{$b$-region} is a connected component of $\Sigma^b$.
        A \emph{region} is either an $a$-region, a $b$-region, or a vertex disc.
        A \emph{transition arc} is a connected component of the $1$-dimensional submanifold $\Gamma$ discussed above, while a \emph{boundary arc} is an edge of $\partial\Sigma$ that bounds a stripe.
        Recall that each boundary arc has an index $i\in\left\{1,\dots,\ell\right\}$ such that the arc is labelled by the letter $a_i\in\left\{a^{\pm1}\right\}$ (or $b_i\in\left\{b^{\pm1}\right\}$) of the word $\Phi(g)$.
        Note that a transition arc might be a loop, but a boundary arc cannot because $\Phi(g)$ must be a cyclically reduced alternating word starting and ending in distinct letters since $\Phi\left(g\right)^n$ is alternating for all $n$.
        
        Observe that each $a$- or $b$-region is a subsurface of $\Sigma$, and each of its boundary components alternates between transition and boundary arcs.
        We put one vertex at each endpoint of each transition or boundary arc --- those vertices are called \emph{arc endpoints} --- and an additional vertex in the interior of each transition arc --- those are called \emph{transition vertices}.
        In the exceptional case of a transition arc that is a loop, there is one arc endpoint at the unique endpoint, and one transition vertex in the interior of the arc.
        Then we add boundary arcs and half-transition arcs as edges.
        In order to complete this to a cellulation of $\Sigma$, we need to add more edges (until the edge set cuts $\Sigma$ into discs), but in light of Corollary \ref{cor:surf-dec}, the way we do this is irrelevant as we will estimate the total curvature of $\Sigma$ by counting the interior curvature of each region, and the interior curvature does not depend on how a subsurface is subdivided into discs.
        Likewise, it suffices to define the total angle of each vertex inside each region, without specifying how the total angle is split between the faces of $\Sigma$.
        Note that, for any choice of cellulation of $\Sigma$, we can pick angles for all corners in the cellulation with specified total angles in the regions of $\Sigma$ --- indeed, such a choice amounts to solving a system of affine equations where each variable --- corresponding to an angle --- appears in exactly one equation --- corresponding to the total angle of the associated vertex in the relevant region.
        
        The angle structure is now defined as follows.
        
        Each arc endpoint $v$ is contained in two regions (on both sides of the corresponding transition arc).
        The total angle of $v$ in each of these regions is defined to be
        \begin{itemize}
            \item A right angle ($\pi/2$) if $v\in\partial\Sigma$, or
            \item A flat angle ($\pi$) if $v\not\in\partial\Sigma$.
        \end{itemize}
        
        Consider a transition vertex $v$.
        It lies in two regions, at most one of which is a vertex disc.
        Consider an $a$- or $b$-region $\Lambda$ on the boundary of which $v$ lies.
        Recall that the boundary of $\Lambda$ alternates between transition and boundary arcs; $v$ lies on a transition arc, which is preceded by a boundary arc with index $i\in\left\{1,\dots,\ell\right\}$, and succeeded by another boundary arc with index $j\in\left\{1,\dots,\ell\right\}$ --- here, the order is induced by the orientation that $\Lambda$ inherits from $\Sigma$.
        We define the total angle of $v$ in the $a$- or $b$-region $\Lambda$ to be
        \begin{equation}\label{eq:def-angle}
            \angle_{\tot}^\Lambda(v)\coloneqq\theta_{i,j}\coloneqq\begin{cases}2\pi\quad&\textrm{if $i\geq j$}\\0\quad&\textrm{if $i<j$}\end{cases}.
        \end{equation}
        In the exceptional case where $v$ is the endpoint of a transition arc which is a loop, there is no preceding or succeeding boundary arc, and we set $\angle_{\tot}^\Lambda(v)\coloneqq\pi$.
        
        Figure \ref{fig:ang-str-reg} shows an annular $a$-region and the definition of the angle structure in that region.
        
        If the other region in which $v$ lies is also an $a$- or $b$-region, then its total angle in that region is defined in the same way.
        Otherwise, $v$ lies in a vertex disc $\Delta$, and its total angle in $\Delta$ is defined to be
        \begin{equation}\label{eq:def-angle-2}
            \angle_{\tot}^\Delta(v)\coloneqq2\pi-\angle_{\tot}^\Lambda(v).
        \end{equation}
        
        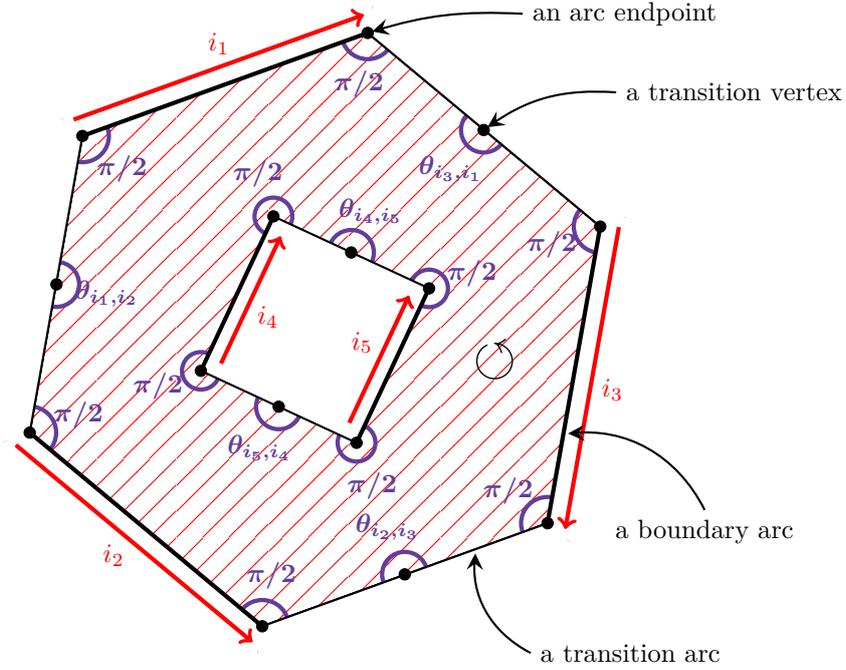
\begin{figure}[hbt]
            \centering
            \begin{tikzpicture}[every node/.style={rectangle,draw=none,fill=none},scale=4,rotate=20]
                \coordinate (o) at (0,0);
                \foreach \i in {0,...,5}{
                    \coordinate (a\i) at ({\i*360/6}:1cm);
                    \coordinate (ap\i) at ({\i*360/6}:1.2cm);
                    \coordinate (aang\i) at ({\i*360/6}:.83cm);}
                \foreach \i in {0,...,3}{
                    \coordinate (b\i) at ({\i*360/4}:.4cm);
                    \coordinate (bang\i) at ({\i*360/4}:.55cm);}
                \foreach \i [evaluate={\j=int(mod(\i+1,6))}] in {0,...,5}{
                    \coordinate (a\i\j) at (barycentric cs:a\i=1,a\j=1);
                    \coordinate (aang\i\j) at (barycentric cs:a\i\j=1,o=.24);}
                \foreach \i [evaluate={\j=int(mod(\i+1,4))}] in {0,...,3}{
                    \coordinate (b\i\j) at (barycentric cs:b\i=1,b\j=1);
                    \coordinate (bang\i\j) at (barycentric cs:b\i\j=1,o=-.35);}
                
                \draw [draw,pattern={Lines[angle=\angred,distance=7]},pattern color=red]
                    (a0) \foreach \i in {1,...,5}{-- (a\i)} -- cycle;
                \draw [draw,thick]
                    (a0) \foreach \i in {1,...,5}{-- (a\i)} -- cycle;
                    
                \foreach \i [evaluate={\j=int(mod(\i+1,6))},evaluate={\u=int(mod(\i,2))}] in {0,...,5}{
                    \draw node [draw,circle,ultra thick,fill=none,inner sep=7pt,RoyalPurple] at (a\i) {};
                    \ifnum\u=0
                        \draw node [draw,circle,ultra thick,fill=none,inner sep=6pt,RoyalPurple] at (a\i\j) {};
                    \fi}
                
                \foreach \i [evaluate={\j=int(mod(\i+1,6))}] in {0,...,5}{
                    \draw [draw=none,fill=white] (a\i) -- (a\j) -- (ap\j) -- (ap\i) -- cycle;}
                    
                \foreach \i [evaluate={\j=int(mod(\i+1,4))},evaluate={\u=int(mod(\i,2))}] in {0,...,3}{
                    \draw node [draw,circle,ultra thick,fill=none,inner sep=5pt,RoyalPurple] at (b\i) {};
                    \ifnum\u=0
                        \draw node [draw,circle,ultra thick,fill=none,inner sep=6pt,RoyalPurple] at (b\i\j) {};
                    \fi}
                    
                \draw [draw,thick,fill=none]
                    (a0) \foreach \i in {1,...,5}{-- (a\i)} -- cycle;
                \draw [draw,thick,fill=white]
                    (b0) \foreach \i in {1,...,3}{-- (b\i)} -- cycle;
                
                \foreach \i [evaluate={\j=int(mod(\i+1,6))}] in {1,3,5}{
                    \draw [ultra thick] (a\i) -- (a\j);
                    \ifnum\i=3
                        \draw [ultra thick,red,->,shorten <= 3pt, shorten >= 3pt]
                            ([xshift={2*cos(\i*360/6)},yshift={2*sin(\i*360/6)}]a\i) -- ([xshift={2*cos(\j*360/6)},yshift={2*sin(\j*360/6)}]a\j)
                            coordinate [midway,xshift={9*cos((\i+.5)*360/6)},yshift={9*sin((\i+.5)*360/6)}] (ma\i);
                    \else
                        \draw [ultra thick,red,<-,shorten <= 3pt, shorten >= 3pt]
                            ([xshift={2*cos(\i*360/6)},yshift={2*sin(\i*360/6)}]a\i) -- ([xshift={2*cos(\j*360/6)},yshift={2*sin(\j*360/6)}]a\j)
                            coordinate [midway,xshift={9*cos((\i+.5)*360/6)},yshift={9*sin((\i+.5)*360/6)}] (ma\i);
                    \fi}
                
                \foreach \i [evaluate={\j=int(mod(\i+1,4))}] in {1,3}{
                    \draw [ultra thick] (b\i) -- (b\j);
                    \ifnum\i=3
                        \draw [ultra thick,red,->]
                            ([xshift={-2*cos(\i*360/4)},yshift={-2*sin(\i*360/4)}]b\i) -- ([xshift={-2*cos(\j*360/4)},yshift={-2*sin(\j*360/4)}]b\j)
                            coordinate [midway,xshift={-9*cos((\i+.5)*360/4)},yshift={-9*sin((\i+.5)*360/4)}] (mb\i);
                    \else
                        \draw [ultra thick,red,<-]
                            ([xshift={-2*cos(\i*360/4)},yshift={-2*sin(\i*360/4)}]b\i) -- ([xshift={-2*cos(\j*360/4)},yshift={-2*sin(\j*360/4)}]b\j)
                            coordinate [midway,xshift={-9*cos((\i+.5)*360/4)},yshift={-9*sin((\i+.5)*360/4)}] (mb\i);
                    \fi}
                    
                \draw node [red] at (ma1) {${i_1}$};
                \draw node [red] at (ma3) {${i_2}$};
                \draw node [red] at (ma5) {${i_3}$};
                \draw node [red] at (mb1) {${i_4}$};
                \draw node [red] at (mb3) {${i_5}$};
                
                \foreach \i [evaluate={\j=int(mod(\i+1,6))},evaluate={\u=int(mod(\i,2))}] in {0,...,5}{
                    \draw node [draw,circle,fill,inner sep=1.5pt] at (a\i) {};
                    \draw node [RoyalPurple] at (aang\i) {$\bm{\pi/2}$};
                    \ifnum\u=0
                        \draw node [draw,circle,fill,inner sep=1.5pt] at (a\i\j) {};
                    \fi}
                    
                \foreach \i [evaluate={\j=int(mod(\i+1,4))},evaluate={\u=int(mod(\i,2))}] in {0,...,3}{
                    \draw node [draw,circle,fill,inner sep=1.5pt] at (b\i) {};
                    \draw node [RoyalPurple] at (bang\i) {$\bm{\pi/2}$};
                    \ifnum\u=0
                        \draw node [draw,circle,fill,inner sep=1.5pt] at (b\i\j) {};
                    \fi}
                    
                \draw node [RoyalPurple] at (aang01) {$\bm{\theta_{i_3,i_1}}$};
                \draw node [RoyalPurple] at (aang23) {$\bm{\theta_{i_1,i_2}}$};
                \draw node [RoyalPurple] at (aang45) {$\bm{\theta_{i_2,i_3}}$};
                \draw node [RoyalPurple] at (bang01) {$\bm{\theta_{i_4,i_5}}$};
                \draw node [RoyalPurple] at (bang23) {$\bm{\theta_{i_5,i_4}}$};
                \draw node at (-30:.6) {\Huge${\circlearrowleft}$};
                        
                \begin{scope}[decoration={markings,mark=at position 5pt with      {\arrow[scale=1.5,>=stealth,xshift=2.5\pgflinewidth]{<}}}]
                    \draw [thick,postaction=decorate,shorten <= 3pt] (a1) to [bend left=10] (1,.75)
                        node [right] {an arc endpoint};
                    \draw [thick,postaction=decorate,shorten <= 3pt] (a01) to [bend left=20] (1.2,.4)
                        node [right] {a transition vertex};
                    \draw [thick,postaction=decorate,shorten <= 3pt] (barycentric cs:a45=1,a5=1) to [bend right=35] (.3,-1.25)
                        node [right] {a transition arc};
                    \draw [thick,postaction=decorate,shorten <= 3pt] (barycentric cs:a5=.7,a0=.3) to [bend left=35] (1,-1)
                        node [below] {a boundary arc};
                \end{scope}
            \end{tikzpicture}
            \caption{The angle structure in an $a$-region. Boundary arcs are in bold, with labels indicated by parallel arrows. Angles are indicated in purple.}\label{fig:ang-str-reg}
        \end{figure}
            
        \begin{rmk}\label{rmk:lp-dual}
            There is a strong connection between our method and Chen's linear programming duality \cites{chen,chen-specgaps,chen-heuer,chen:kervaire}.
            To expand on this connection, let us compare the above construction with the overview that Chen gives of his method in \cite{chen:kervaire}*{\S{}4}.
            Chen's notion of \emph{turns} corresponds to our transition arcs; they carry a \emph{cost}, which corresponds to the angles $\theta_{i,j}$ defined above.
            Chen's \emph{pieces} correspond to our regions, and the cost of a piece corresponds to the curvature in our language.
            Lemmas 3.6 and 4.1 of \cite{chen:kervaire} are essentially based on the Gauß--Bonnet formula.
            
            This analogy can be used to translate Chen's LP duality proofs into our language --- for example, Chen's spectral gap for free products \cite{chen-specgaps} is reproved in the author's PhD thesis \cite{m:phd}*{\S{}VI.2}.
            However, the angle structure method might be more flexible and does not require one to first formulate a linear programming problem that computes $\scl$.
        \end{rmk}
        
        \subsection{Boundary orientation of \texorpdfstring{$a$}{a}- and \texorpdfstring{$b$}{b}-regions}
        
        In order to estimate the interior curvature of $a$- and $b$-regions, we need a better understanding of their boundary labelling.
        Each boundary arc of $\Sigma$ is labelled by a letter in $\left\{a^{\pm1},b^{\pm1}\right\}$.
        Another way to see this labelling is as an orientation of the boundary arc, as indicated by bold arrows in Figures \ref{fig:new-discs}, \ref{fig:zebra}, \ref{fig:cor-bnd-lab}, \ref{fig:ex:stripeptrn}, \ref{fig:unzip}, \ref{fig:ang-str-reg} --- the label of the boundary arc is $a^{+1}$ (or $b^{+1}$) if its orientation matches that of $\Sigma$, and $a^{-1}$ (or $b^{-1}$) otherwise.
        We say that two boundary arcs have \emph{opposite orientations} if one of them is labelled by $a^{+1}$ (or $b^{+1}$) while the other one is labelled by $a^{-1}$ (or $b^{-1}$) --- in other words, one is oriented consistently with $\Sigma$, and the other one isn't.
        
        \begin{lmm}\label{lmm:bnd-labels}
            Let $\Lambda$ be an $a$- or $b$-region of $\Sigma$ which is topologically a disc and whose boundary contains at least two distinct boundary arcs.
            Then $\partial\Lambda$ contains two boundary arcs with opposite orientations.
        \end{lmm}
        \begin{proof}
            The region $\Lambda$ is built out of stripes and hexagons.
            Each stripe has four boundary edges, alternating between sections of transition arcs, and edges that have an orientation given by whether they map positively to $S^1_a$ or $S^1_b$.
            Those edges are called \emph{directed edges}.
            The orientations of directed edges are indicated by bold arrows in Figures \ref{fig:new-discs}, \ref{fig:zebra}, \ref{fig:cor-bnd-lab}, \ref{fig:ex:stripeptrn}, \ref{fig:unzip}, \ref{fig:ang-str-reg}.
            Likewise, each hexagon has six boundary edges, alternating between sections of transitions arcs, and directed edges.
            Stripes and hexagons are glued along directed edges, in such a way that orientations of directed edges which are glued together agree.
            
            Now the key observation is that each stripe and each hexagon contains two directed edges with opposite orientations (i.e. one of them matches the orientation of the stripe or hexagon, while the other does not).
            In stripes (Figure \ref{fig:opp-or-str}), this is because parallel directed edges have parallel orientations, so one of them matches the orientation of the stripe, and the other does not.
            In hexagons (Figure \ref{fig:opp-or-hex}), one must come back to the construction of \S{}\ref{subsec:ext-cell} to see that the labels of the boundary edges correspond to letters $x_1,x_2,x_3$ as in the definition of letter-quasimorphisms (Definition \ref{def:let-qm}), and so in particular $x_1x_2x_3\in\left\{a^{\pm1},b^{\pm1}\right\}$.
            This means that each hexagon has two directed edges with opposite orientations.
        
            \begin{figure}[htb]
                \centering
                \begin{subfigure}[c]{13em}
                    \centering
                    \begin{tikzpicture}[every node/.style={rectangle,draw=none,fill=none},rotate=-70,scale=1]
                        \coordinate (a0) at (0,0);
                        \coordinate (a1) at (2.25,0);
                        \coordinate (b0) at (0,4);
                        \coordinate (b1) at (2.25,4);
                        \draw [pattern={Lines[angle=\angred,distance=\spr]},pattern color=red] (a0) -- (a1) -- (b1) -- (b0) -- cycle;
                        \draw [thick] (a0) -- (a1) -- (b1) -- (b0) -- cycle;
                        \draw [ultra thick] (a0) -- (a1) (b1) -- (b0);
                        \draw [ultra thick,red,shorten <=2pt,shorten >=2pt,->]
                            ([yshift=5pt]a0) -- ([yshift=5pt]a1);
                        \draw [ultra thick,red,shorten <=2pt,shorten >=2pt,->]
                            ([yshift=-5pt]b0) -- ([yshift=-5pt]b1);
                        \node at (barycentric cs:a0=1,a1=1,b0=1,b1=1) {\Huge$\circlearrowleft$};
                    \end{tikzpicture}
                    \caption{Stripe.}\label{fig:opp-or-str}
                \end{subfigure}
                \hspace{3em}
                \begin{subfigure}[c]{13em}
                    \centering
                    \begin{tikzpicture}[every node/.style={rectangle,draw=none,fill=none},rotate=25,scale=1.8]
                        \foreach \i in {0,...,5}{
                            \coordinate (a\i) at ({\i*360/6}:1cm);
                            \coordinate (ap\i) at ({\i*360/6}:.9cm);}
                        \draw [pattern={Lines[angle=\angred,distance=\spr]},pattern color=red] (a0) \foreach \i in {1,...,5}{-- (a\i)} -- cycle;
                        \draw [thick] (a0) \foreach \i in {1,...,5}{-- (a\i)} -- cycle;
                        \foreach \i [evaluate={\j={\i+1}}] in {0,2}{
                            \draw [ultra thick] (a\i) -- (a\j);}
                        \draw [ultra thick,red,->] (ap0) -- (ap1);
                        \draw [ultra thick,red,<-] (ap2) -- (ap3);
                        \draw [ultra thick,red,->] (ap4) -- (ap5);
                        \node at (barycentric cs:a0=1,a2=1,a4=1) {\Huge$\circlearrowleft$};
                    \end{tikzpicture}
                    \caption{Hexagon.}\label{fig:opp-or-hex}
                \end{subfigure}
                \caption{Directed edges with opposite orientations (highlighted in bold) in stripes and hexagons.}\label{fig:opp-or-str-hex}
            \end{figure}
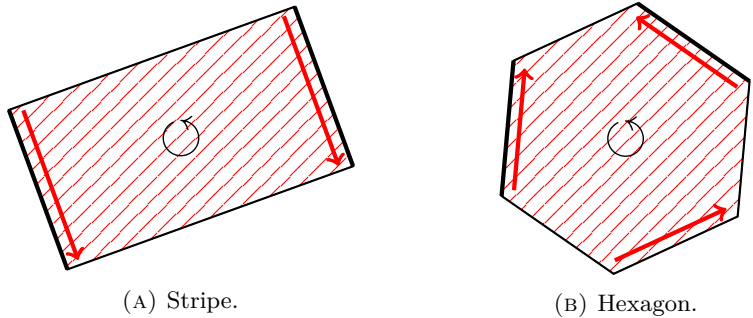
            
            Using this observation, and the fact that stripes and hexagons are glued so that directed edges have matching orientations, we can construct two boundary arcs of $\Lambda$ with opposite orientations as follows.
            Start at a boundary arc of $\Lambda$; this arc is a directed edge $e_1$ of a stripe or a hexagon, which we denote by $\Delta_1$.
            We have observed that $\partial\Delta_1$ has a directed edge $e_2$ with orientation opposite to $e_1$; pick a point $p_1$ in the interior of $e_1$ and a point $p_2$ in the interior of $e_2$ and connect them via an arc in the interior of $\Delta_1$.
            Now $e_2$ borders another stripe or hexagon, which we denote by $\Delta_2$; extend the previous arc to a point $p_3$ in the interior of a directed edge $e_3$ of $\Delta_3$ with orientation opposite to $e_2$.
            Iterate this process; it will follow from the following claim that we cannot visit the same stripe or hexagon twice:
            \begin{claim*}
                There is no sequence of at least two successively adjacent stripes and hexagons in $\Lambda$ (without backtracking) starting and finishing on the same stripe or hexagon.
            \end{claim*}
            \begin{proof}[Proof of the claim]
                Suppose for contradiction that there is a sequence of stripes and hexagons as above in $\Lambda$.
                Draw a loop $\beta$ in $\Lambda$ connecting the midpoints of the successive connecting edges of this sequence.
                We will show that $\beta$ is non-contractible (in $\Lambda$), contradicting the assumption that $\Lambda$ is topologically a disc.
                Suppose otherwise: hence, there is a homotopy $\beta_\bullet:[0,1]^2\rightarrow\Lambda$ from $\beta=\beta_0$ to a constant loop $\beta_1$.
                Note that the restriction of $\beta$ to each stripe or hexagon enters and exits using two opposite edges; since those are separated by edges contained in $\partial\Lambda$, this must also be true of each $\beta_t$ (for $t\in[0,1]$).
                Therefore, the number of stripes and hexagons visited by $\beta_t$ does not depend on $t$.
                But this number is at least $2$ for $\beta_0$, while it is exactly $1$ for $\beta_1$; this is a contradiction.
            \end{proof}
            Moreover, $\Lambda$ is compact, so the above process eventually reaches a boundary arc of $\Lambda$.
            The construction is illustrated in Figure \ref{fig:constr-opp-or}.
            \begin{figure}[htb]
                \centering
                \def\r{.8}
                \begin{tikzpicture}[every node/.style={rectangle,draw=none,fill=none},scale=1.5,rotate=20]
                    \foreach \i in {0,...,5}{
                        \coordinate (b\i) at ({2.5*cos(120)+\r*cos(360*\i/6+30)},{2.5*sin(120)+\r*sin(360*\i/6+30)});
                        \coordinate (bp\i) at ({2.5*cos(120)+.9*\r*cos(360*\i/6+30)},{2.5*sin(120)+.9*\r*sin(360*\i/6+30)});
                        \coordinate (c\i) at ({\r*cos(360*\i/6+30)},{\r*sin(360*\i/6+30)});
                        \coordinate (cp\i) at ({.9*\r*cos(360*\i/6+30)},{.9*\r*sin(360*\i/6+30)});
                        \coordinate (d\i) at ({3+\r*cos(360*\i/6+30)},{\r*sin(360*\i/6+30)});
                        \coordinate (dp\i) at ({3+.9*\r*cos(360*\i/6+30)},{.9*\r*sin(360*\i/6+30)});}
                        
                    \coordinate (a0) at ([xshift=-3em]b2);
                    \coordinate (a1) at ([xshift=-3em]b3);
                    \coordinate (ap0) at ([yshift=1em]a0);
                    \coordinate (ap1) at ([yshift=-1em]a1);
                    
                    \coordinate (e0) at ([xshift=1.5em,yshift=2.598em]d0);
                    \coordinate (e1) at ([xshift=1.5em,yshift=2.598em]d1);
                    \coordinate (ep0) at ([xshift=1.299em,yshift=-.75em]e0);
                    \coordinate (ep1) at ([xshift=-1.299em,yshift=.75em]e1);
                    
                    \draw [pattern={Lines[angle=\anggrn,distance=\spr]},pattern color=PineGreen] (a0) -- (a1) -- (b3) -- (b2) -- cycle;
                    \draw [thick] (a0) -- (a1) -- (b3) -- (b2) -- cycle;
                    \draw [pattern={Lines[angle=\anggrn,distance=\spr]},pattern color=PineGreen] (b0) \foreach \i in {1,...,5}{-- (b\i)} -- cycle;
                    \draw [thick] (b0) \foreach \i in {1,...,5}{-- (b\i)} -- cycle;
                    \draw [pattern={Lines[angle=\anggrn,distance=\spr]},pattern color=PineGreen] (b4) -- (c2) -- (c1) -- (b5) -- cycle;
                    \draw [thick] (b4) -- (c2) -- (c1) -- (b5) -- cycle;
                    \draw [pattern={Lines[angle=\anggrn,distance=\spr]},pattern color=PineGreen] (c0) \foreach \i in {1,...,5}{-- (c\i)} -- cycle;
                    \draw [thick] (c0) \foreach \i in {1,...,5}{-- (c\i)} -- cycle;
                    \draw [pattern={Lines[angle=\anggrn,distance=\spr]},pattern color=PineGreen] (c0) -- (c5) -- (d3) -- (d2) -- cycle;
                    \draw [thick] (c0) -- (c5) -- (d3) -- (d2) -- cycle;
                    \draw [pattern={Lines[angle=\anggrn,distance=\spr]},pattern color=PineGreen] (d0) \foreach \i in {1,...,5}{-- (d\i)} -- cycle;
                    \draw [thick] (d0) \foreach \i in {1,...,5}{-- (d\i)} -- cycle;
                    \draw [pattern={Lines[angle=\anggrn,distance=\spr]},pattern color=PineGreen] (e0) -- (e1) -- (d1) -- (d0) -- cycle;
                    \draw [thick] (e0) -- (e1) -- (d1) -- (d0) -- cycle;
                    
                    \draw [ultra thick] (ap0) -- (ap1) node [below] {$\partial\Sigma$};
                    \draw [ultra thick] (ep0) -- (ep1) node [left] {$\partial\Sigma$};
                    
                    \draw [ultra thick,PineGreen,->] (bp2) -- (bp3);                    
                    \draw [ultra thick,PineGreen,->] (bp5) -- (bp4);
                    \draw [ultra thick,PineGreen,->] (cp1) -- (cp2);
                    \draw [ultra thick,PineGreen,->] (cp0) -- (cp5);
                    \draw [ultra thick,PineGreen,->] (dp2) -- (dp3);
                    \draw [ultra thick,PineGreen,->] (dp1) -- (dp0);
                    
                    \begin{scope}[ultra thick,PineGreen,shorten <=1.75pt,shorten >=1.75pt,->]
                        \draw ([xshift=2pt]a0) -- ([xshift=2pt]a1);
                        \draw ([xshift=-2pt]b2) -- ([xshift=-2pt]b3);
                        \draw ([xshift=-2*.5pt,yshift=2*.866025404pt]c1) -- ([xshift=-2*.5pt,yshift=2*.866025404pt]c2);
                        \draw ([xshift=2*.5pt,yshift=-2*.866025404pt]b5) -- ([xshift=2*.5pt,yshift=-2*.866025404pt]b4);
                        \draw ([xshift=2pt]c0) -- ([xshift=2pt]c5);
                        \draw ([xshift=-2pt]d2) -- ([xshift=-2pt]d3);
                        \draw ([xshift=2*.5pt,yshift=2*.866025404pt]d1) -- ([xshift=2*.5pt,yshift=2*.866025404pt]d0);
                        \draw ([xshift=-2*.5pt,yshift=-2*.866025404pt]e1) -- ([xshift=-2*.5pt,yshift=-2*.866025404pt]e0);
                    \end{scope}
                    
                    \draw node[circle,draw,fill=black,inner sep=1.5pt,MidnightBlue] (p1) at (barycentric cs:a0=1,a1=1) {};
                    \draw node[circle,draw,fill=black,inner sep=1.5pt,MidnightBlue] (p2) at (barycentric cs:b2=1,b3=1) {};
                    \draw node[circle,draw,fill=black,inner sep=1.5pt,MidnightBlue] (p3) at (barycentric cs:b5=1,b4=1) {};
                    \draw node[circle,draw,fill=black,inner sep=1.5pt,MidnightBlue] (p4) at (barycentric cs:c1=1,c2=1) {};
                    \draw node[circle,draw,fill=black,inner sep=1.5pt,MidnightBlue] (p5) at (barycentric cs:c0=1,c5=1) {};
                    \draw node[circle,draw,fill=black,inner sep=1.5pt,MidnightBlue] (p6) at (barycentric cs:d2=1,d3=1) {};
                    \draw node[circle,draw,fill=black,inner sep=1.5pt,MidnightBlue] (p7) at (barycentric cs:d1=1,d0=1) {};
                    \draw node[circle,draw,fill=black,inner sep=1.5pt,MidnightBlue] (p8) at (barycentric cs:e1=1,e0=1) {};
                    
                    \draw [ultra thick,MidnightBlue] (p1) -- (p2) -- (p3) -- (p4) -- (p5) -- (p6) -- (p7) -- (p8);
                    
                    \draw [dashed,thick,MidnightBlue,<-,shorten <= 2pt] (p1) to [bend left=30] (-2.5,1.6) node [below] {$p_1$};
                    \draw [dashed,thick,MidnightBlue,<-,shorten <= 2pt] (p2) to [bend right=40] (-2.3,2.9) node [left] {$p_2$};
                    \draw [dashed,thick,MidnightBlue,<-,shorten <= 2pt] (p3) to [bend left=20] (0,1.7) node [right] {$p_3$};
                    \draw [dashed,thick,MidnightBlue,<-,shorten <= 2pt] (barycentric cs:p3=1,p4=1) to [bend left=30] (-1.5,1) node [left] {$\eta$};
                    \draw node [circle,fill=white,draw=none,inner sep=.5pt] at (-.15,-.15) {\color{PineGreen}{$\Lambda$}};
                    \draw node at (3.15,-.15) {\Huge${\circlearrowleft}$};
                \end{tikzpicture}
                \caption{Construction of a path connecting two boundary arcs with opposite orientations in an $a$- or $b$-region.}\label{fig:constr-opp-or}
            \end{figure}
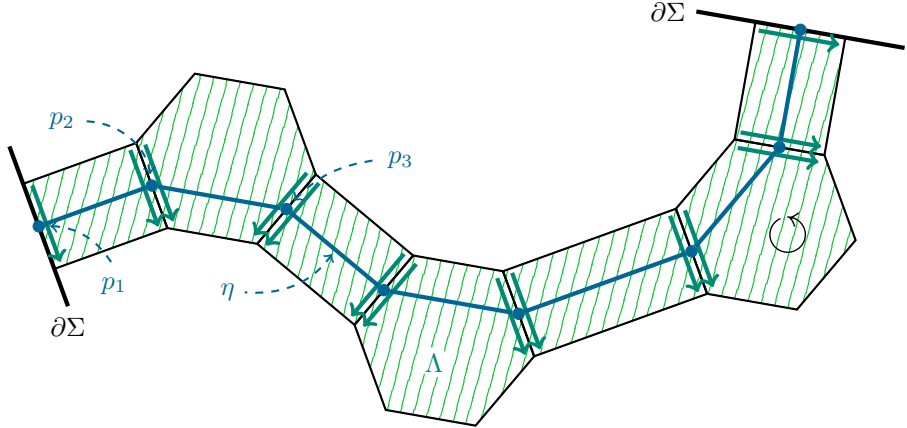
            Hence, we obtain an arc $\eta$ visiting an injective sequence of stripes and hexagons, with endpoints on distinct boundary arcs of $\Lambda$, with the property that all directed edges of stripes or hexagons crossed by $\eta$ have the same transverse orientation with respect to $\eta$.
            In particular, the boundary arcs at the extremities of $\eta$ must have opposite orientations (with respect to the orientation of $\Lambda$).
        \end{proof}
        
        \subsection{Estimating the curvature}
    
        We will estimate the total curvature of $\Sigma$ using its decomposition into $a$-regions, $b$-regions, and vertex discs.
        Corollary \ref{cor:surf-dec} says that it suffices to estimate the curvature of each vertex and the interior curvature of each region.
        
        \begin{claim}
            Each arc endpoint of $\Sigma$ has zero curvature.
        \end{claim}
        \begin{proof}[Proof of the claim]
            Let $v$ be an arc endpoint.
            \begin{itemize}
                \item Suppose that $v$ lies at the endpoint of a transition arc and a boundary arc.
                In particular, $v\in\partial\Sigma$, and $v$ is contained in two regions. The angle of $v$ in each region is $\pi/2$, and $\chi\left(\lk_\Sigma(v)\right)=1$, so that $\kappa(v)=0$.
                \item If $v$ is only the endpoint of a transition arc (which has to be a loop again), then $v\not\in\partial\Sigma$, $v$ is contained in two regions, and its angle is $\pi$ in each.
                Since $\chi\left(\lk_\Sigma(v)\right)=0$ in this case, we get $\kappa(v)=0$ again.\qedhere
            \end{itemize}
        \end{proof}
        
        \begin{claim}
            Each $a$- or $b$-region of $\Sigma$ has non-positive interior curvature.
        \end{claim}
        \begin{proof}[Proof of the claim]
            Let $\Lambda$ be an $a$- or $b$-region of $\Sigma$.
            Recall from \S{}\ref{sec:ang-str} that the interior curvature of $\Lambda$ is
            \[
                \kappa_{\intr}(\Lambda)=2\pi\cdot\chi(\Lambda)+\sum_{v\in V(\Lambda)}\left(\angle_{\tot}^\Lambda(v)-\pi\right).
            \]
            We first show that the sum $S\coloneqq\sum_{v}\left(\angle_{\tot}^\Lambda(v)-\pi\right)$ is non-positive.
            To do so, observe that each vertex of $\Lambda$ lies on a unique transition arc.
            Therefore, for the purpose of computing the sum $S$, we can partition $V(\Lambda)$ into subsets corresponding to transition arcs.
            A transition arc which is a loop contains two vertices of $\Lambda$, each of which has total angle $\pi$ in $\Lambda$, so that its contribution to $S$ is zero.
            Each transition arc which is not a loop contains three vertices of $\Lambda$: two arc endpoints with total angle $\pi/2$ in $\Lambda$ each, and one transition vertex with total angle $0$ or $2\pi$ in $\Lambda$ --- see Figure \ref{fig:ang-str-reg}.
            Hence, the total contribution of a transition arc to $S$ is at most $\left(\frac{\pi}{2}-\pi\right)+\left(\frac{\pi}{2}-\pi\right)+\left(2\pi-\pi\right)=0$.
            It follows that $S\leq0$.
            This proves that $\kappa_{\intr}(\Lambda)\leq0$ as soon as $\chi(\Lambda)\leq0$.
            
            It remains to analyse the case where $\chi(\Lambda)>0$.
            Since $\Lambda$ has nonempty boundary, the assumption that $\chi(\Lambda)>0$ implies that $\Lambda$ is a disc, and $\chi(\Lambda)=1$.
            Looking back at the previous argument, the total contribution to the sum $S$ of a transition arc preceded by a boundary arc with index $i\in\left\{1,\dots,\ell\right\}$ and succeeded by a boundary arc with index $j\in\left\{1,\dots,\ell\right\}$ is exactly $\left(\theta_{i,j}-2\pi\right)$.
            Now let $i_1,\dots,i_k$ be the successive indices of the boundary arcs of $\Lambda$, ordered according to the orientation of $\Lambda$.
            Then we have
            \[
                \kappa_{\intr}(\Lambda)=2\pi+\sum_{j=1}^k\left(\theta_{i_j,i_{j+1}}-2\pi\right),
            \]
            with the convention that $i_{k+1}=i_1$.
            If the indices $i_1,\dots,i_k$ are not all equal, then there is some $j_0$ such that $i_{j_0}<i_{j_0+1}$, and so $\left(\theta_{i_{j_0},i_{j_0+1}}-2\pi\right)=-2\pi$.
            Since $\theta_{i_j,i_{j+1}}\leq2\pi$ for all $j$ by definition, it then follows that
            \[
                \kappa_{\intr}(\Lambda)\leq2\pi+\left(\theta_{i_{j_0},i_{j_0+1}}-2\pi\right)=0.
            \]
            
            Therefore, the only case where $\kappa_{\intr}(\Lambda)$ can be positive is if $\Lambda$ is topologically a disc and all its boundary arcs have the same index $i$.
            In particular, all the boundary arcs are labelled by the same letter (either $a$, $a^{-1}$, $b$, or $b^{-1}$), which contradicts Lemma \ref{lmm:bnd-labels}.
        \end{proof}
        
        \begin{claim}
            The combined contribution of vertex discs and transition vertices to the total curvature of $\Sigma$ is at most $-2\pi\cdot n(\Sigma)$.
        \end{claim}
        \begin{proof}[Proof of the claim]
            Let us start with transition vertices.
            Each transition vertex $v$ lies on a transition arc $\alpha$, bounding two distinct regions of $\Sigma$.
            If $\alpha$ bounds a vertex disc $\Delta$ and an $a$- or $b$-region $\Lambda$, then by $\left(\ref{eq:def-angle-2}\right)$, we have
            \[
                \kappa(v)=2\pi-\angle_{\tot}^\Delta(v)-\angle_{\tot}^\Lambda(v)=0.
            \]
            Otherwise, $\alpha$ bounds an $a$-region $\Lambda_a$ and a $b$-region $\Lambda_b$.
            Let $i$ and $j$ be the indices of the boundary arcs preceding and succeeding $\alpha$ in $\partial\Lambda_a$.
            Hence, the indices of the boundary arcs preceding and succeeding $\alpha$ in $\partial\Lambda_b$ must be $j-1$ and $i$ (with indices taken modulo $\ell$), as in Figure \ref{fig:trans-vrtx}.
            
            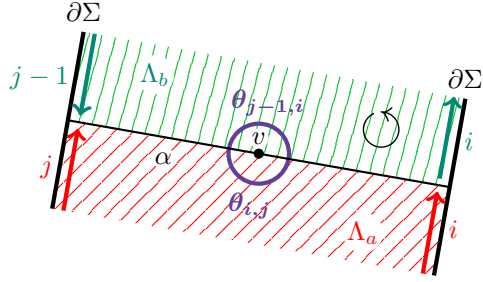
\begin{figure}[htb]
                \centering
                \begin{tikzpicture}[every node/.style={rectangle,draw=none,fill=none},scale=1.7,rotate=80]
                    \foreach \i [evaluate={\x={.7*(\i-1)}}] in {0,1,2}{
                        \coordinate (a\i) at (\x,0);
                        \coordinate (b\i) at (\x,3);}
                    \coordinate (v) at (barycentric cs:a1=1,b1=1);
                    
                    \draw [draw=none,pattern={Lines[angle=\angred,distance=\spr]},pattern color=red]
                        (a0) -- (a1) -- (b1) -- (b0) -- cycle;
                    \draw [draw=none,pattern={Lines[angle=\anggrn,distance=\spr]},pattern color=PineGreen]
                        (a2) -- (a1) -- (b1) -- (b2) -- cycle;
                    
                    \draw node [draw,circle,ultra thick,fill=none,inner sep=8pt,RoyalPurple] at (v) {};
                    
                    \draw [ultra thick] (a0) -- (a2) node [above] {$\partial\Sigma$}
                        (b0) -- (b2) node [above] {$\partial\Sigma$};
                    \draw [thick] (a1) -- (b1);
                    
                    \begin{scope}[ultra thick,shorten >=2.5pt]
                        \draw [red,->] ([yshift=2.5pt]a0) -- ([yshift=2.5pt]a1)
                            node [midway,right,inner sep=6pt] {$i$};
                        \draw [red,->] ([yshift=-2.5pt]b0) -- ([yshift=-2.5pt]b1)
                            node [midway,left,inner sep=6pt] {$j$};
                        \draw [PineGreen,<-] ([yshift=2.5pt]a2) -- ([yshift=2.5pt]a1)
                            node [midway,right,inner sep=6pt] {$i$};
                        \draw [PineGreen,->] ([yshift=-2.5pt]b2) -- ([yshift=-2.5pt]b1)
                            node [midway,left,inner sep=6pt] {$j-1$};
                    \end{scope}
                    
                    \draw node[circle,draw,fill=black,inner sep=1.2pt,
                        label={[circle,draw=none,fill=white,inner sep=.25pt]above:$v$}] at (v) {};
                    \draw node[
                        label={[circle,draw=none,fill=white,inner sep=.25pt]below:$\alpha$}] at (barycentric cs:a1=.25,b1=.75) {};
                        
                    \draw node at (barycentric cs:a0=1.4,b0=1.4,a1=1,b1=1) {\color{RoyalPurple}{$\bm{\theta_{i,j}}$}};
                    \draw node at (barycentric cs:a2=1.4,b2=1.4,a1=1,b1=1) {\color{RoyalPurple}{$\bm{\theta_{j-1,i}}$}};
                    
                    \draw node [circle,fill=white,draw=none,inner sep=.2pt] at (barycentric cs:a0=8,b0=2,a1=4,b1=1) {\color{red}{$\Lambda_a$}};
                    \draw node [circle,fill=white,draw=none,inner sep=.2pt] at (barycentric cs:a2=2,b2=8,a1=1,b1=4) {\color{PineGreen}{$\Lambda_b$}};
                    \draw node [inner sep=.2pt] at (barycentric cs:a2=4,b2=1,a1=4,b1=1) {\Huge$\circlearrowleft$};
                \end{tikzpicture}
                \caption{Estimating the curvature of transition vertices.}\label{fig:trans-vrtx}
            \end{figure}
            
            It follows that
            \[
                \kappa(v)=2\pi-\angle_{\tot}^{\Lambda_a}(v)-\angle_{\tot}^{\Lambda_b}(v)=2\pi-\theta_{i,j}-\theta_{j-1,i}.
            \]
            If $j\neq1$, then we have $i\geq j$ if and only if $j-1<i$, so that $\theta_{i,j}=2\pi$ if and only if $\theta_{j-1,i}=0$ (see $\left(\ref{eq:def-angle}\right)$), and therefore $\kappa(v)=0$.
            
            On the other hand, if $j=1$, then $j-1=\ell$, and we have $\theta_{i,j}=\theta_{j-1,i}=2\pi$ and $\kappa(v)=-2\pi$.
            
            We now turn our attention to vertex discs.
            By definition, each vertex disc $\Delta$ is a topological disc, with boundary alternating between transition arcs and edges of $\partial\Sigma$ (which are not boundary arcs since they do not bound a stripe).
            The successive transition arcs of $\partial\Delta$ alternatively bound an $a$-region $\Lambda_a^{(1)}$, then a $b$-region $\Lambda_b^{(1)}$, then another $a$-region $\Lambda_a^{(2)}$, etc.
            Here, we are ordering the transition arcs in clockwise order --- i.e. in the order opposite to the orientation of $\partial\Delta$.
            Hence there is an even number of transition arcs, say $2r$, so that the regions adjacent to $\Delta$ are (in clockwise order) $\Lambda_a^{(1)},\Lambda_b^{(1)},\dots,\Lambda_a^{(r)},\Lambda_b^{(r)}$.
            See Figure \ref{fig:cell-disc}, where $r=2$.
            \begin{figure}[htb]
                \centering
                \begin{tikzpicture}[every node/.style={rectangle,draw=none,fill=none},scale=3.8,rotate=0]
                    \foreach \i [evaluate={\t={360*(\i+.5)/8}}] in {1,...,8}{
                        \coordinate (a\i) at ({cos(\t)},{sin(\t)});
                        \coordinate (ap\i) at ({.8*cos(\t)},{.8*sin(\t)});}
                    \foreach \i [evaluate={\t={360*(\i-.5)/4}}] in {1,...,4}{
                        \coordinate (lab\i) at ({.7*cos(\t)},{.7*sin(\t)});}
                    
                    \coordinate (b1) at ({cos(360*1.5/8)+.3},{sin(360*1.5/8)});
                    \coordinate (b2) at ({cos(360*2.5/8)-.3},{sin(360*2.5/8)});
                    \coordinate (b3) at ({cos(360*3.5/8)},{sin(360*3.5/8)+.3});
                    \coordinate (b4) at ({cos(360*4.5/8)},{sin(360*4.5/8)-.3});
                    \coordinate (b5) at ({cos(360*5.5/8)-.3},{sin(360*5.5/8)});
                    \coordinate (b6) at ({cos(360*6.5/8)+.3},{sin(360*6.5/8)});
                    \coordinate (b7) at ({cos(360*7.5/8)},{sin(360*7.5/8)-.3});
                    \coordinate (b8) at ({cos(360*8.5/8)},{sin(360*8.5/8)+.3});
                    
                    \draw [draw=none,pattern={Lines[angle=\anggrn,distance=\spr]},pattern color=PineGreen]
                        (a2) -- (b2) -- (b3) -- (a3) -- cycle;
                    \draw node [circle,fill=white,draw=none,inner sep=.1pt] at (barycentric cs:a2=1,b2=2,a3=1,b3=2) {\color{PineGreen}{$\Lambda_b^{(1)}$}};
                    \draw [draw=none,pattern={Lines[angle=\angred,distance=\spr]},pattern color=red]
                        (a4) -- (b4) -- (b5) -- (a5) -- cycle;
                    \draw node [circle,fill=white,draw=none,inner sep=.1pt] at (barycentric cs:a4=1,b4=2,a5=1,b5=2) {\color{red}{$\Lambda_a^{(1)}$}};
                    \draw [draw=none,pattern={Lines[angle=\anggrn,distance=\spr]},pattern color=PineGreen]
                        (a6) -- (b6) -- (b7) -- (a7) -- cycle;
                    \draw node [circle,fill=white,draw=none,inner sep=.1pt] at (barycentric cs:a6=1,b6=2,a7=1,b7=2) {\color{PineGreen}{$\Lambda_b^{(2)}$}};
                    \draw [draw=none,pattern={Lines[angle=\angred,distance=\spr]},pattern color=red]
                        (a8) -- (b8) -- (b1) -- (a1) -- cycle;
                    \draw node [circle,fill=white,draw=none,inner sep=.1pt] at (barycentric cs:a8=1,b8=2,a1=1,b1=2) {\color{red}{$\Lambda_a^{(2)}$}};
                            
                    \draw [thick,fill=gray!35] (a8) \foreach \i in {1,...,7} {-- (a\i)} -- cycle;
                    \foreach \i [evaluate={\j={\i+1}},evaluate={\h={int(mod(\i+6,8)+1)}}] in {1,3,5,7}{
                        \draw (a\i) -- (a\h) coordinate [midway] (v\i);
                        \draw node[circle,draw,fill=black,inner sep=1.2pt] at (v\i) {};
                        \begin{scope}
                            \clip (a8) \foreach \l in {1,...,7} {-- (a\l)} -- cycle;
                            \draw node [draw,circle,ultra thick,fill=none,inner sep=7pt,RoyalPurple] at (v\i) {};
                            \draw node [draw,circle,ultra thick,fill=none,inner sep=7pt,RoyalPurple] at (a\i) {};
                            \draw node [draw,circle,ultra thick,fill=none,inner sep=7pt,RoyalPurple] at (a\j) {};
                            \draw node at (ap\i) {\color{RoyalPurple}{$\bm{\pi/2}$}};
                            \draw node at (ap\j) {\color{RoyalPurple}{$\bm{\pi/2}$}};
                        \end{scope}
                        \draw [ultra thick] (b\i) -- (b\j);}
                    
                    \begin{scope}[ultra thick,shorten >=5pt]
                        \draw [red,->] ([yshift=-1pt]b1) -- ([yshift=-1pt]a1)
                            node [midway,above,inner sep=6pt] {$i_2$};
                        \draw [PineGreen,->] ([yshift=-1pt]b2) -- ([yshift=-1pt]a2)
                            node [midway,above,inner sep=6pt] {$i_2$};
                        \draw [PineGreen,->] ([xshift=1pt]b3) -- ([xshift=1pt]a3)
                            node [midway,left,inner sep=6pt] {$j_1-1$};
                        \draw [red,<-] ([xshift=1pt]b4) -- ([xshift=1pt]a4)
                            node [midway,left,inner sep=6pt] {$j_1$};
                        \draw [red,<-] ([yshift=1pt]b5) -- ([yshift=1pt]a5)
                            node [midway,below,inner sep=6pt] {$i_1$};
                        \draw [PineGreen,->] ([yshift=1pt]b6) -- ([yshift=1pt]a6)
                            node [midway,below,inner sep=6pt] {$i_1$};
                        \draw [PineGreen,<-] ([xshift=-1pt]b7) -- ([xshift=-1pt]a7)
                            node [midway,right,inner sep=6pt] {$j_2-1$};
                        \draw [red,->] ([xshift=-1pt]b8) -- ([xshift=-1pt]a8)
                            node [midway,right,inner sep=6pt] {$j_2$};
                    \end{scope}
                    
                    \draw node at (lab1) {\color{RoyalPurple}{$\bm{2\pi-\theta_{i_2,j_2}}$}};
                    \draw node at (lab2) {\color{RoyalPurple}{$\bm{2\pi-\theta_{j_1-1,i_2}}$}};
                    \draw node at (lab3) {\color{RoyalPurple}{$\bm{2\pi-\theta_{i_1,j_1}}$}};
                    \draw node at (lab4) {\color{RoyalPurple}{$\bm{2\pi-\theta_{j_2-1,i_1}}$}};
                    
                    \draw node at (-.25,0) {\LARGE$\Delta$};
                    \draw node at (.25,0) {\Huge$\circlearrowleft$};
                    
                    \foreach \i [evaluate={\t={360*\i/4}}] in {1,...,4}{
                        \draw node at ({cos(\t)},{sin(\t)}) {$\partial\Sigma$};}
                \end{tikzpicture}
                \caption{An octagonal vertex disc.}\label{fig:cell-disc}
            \end{figure}
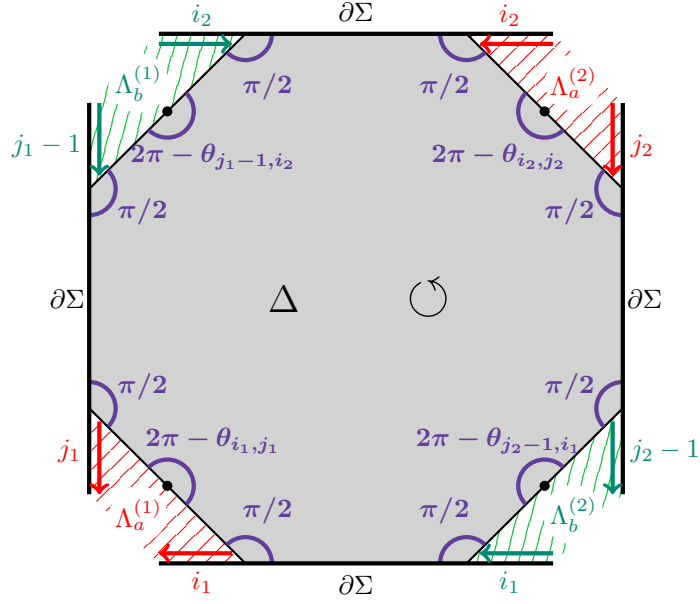
            In $\Lambda_a^{(k)}$, let $i_k$ and $j_k$ be the indices of the boundary arcs preceding and succeeding the transition arc bounding $\Delta$.
            Then in $\Lambda_b^{(k)}$, the indices of the boundary arcs preceding and succeeding the transition arc bounding $\Delta$ must be $j_k-1$ and $i_{k+1}$ (with the convention $i_{r+1}=i_1$).
            Again, we refer the reader to Figure \ref{fig:cell-disc} for an illustration of these notations.
            
            Now $\Delta$ has one transition vertex and two arc endpoints on each transition arc, and each arc endpoint has a total angle of $\pi/2$ in $\Delta$.
            Therefore, the interior curvature of $\Delta$ is
            \begin{equation}\label{eq:kappa-int-cell}
                \kappa_{\intr}(\Delta)=2\pi-\sum_{k=1}^r\left(\theta_{i_k,j_k}+\theta_{j_k-1,i_{k+1}}\right).
            \end{equation}
            The assumption that $\Sigma$ is disc-free implies that $r\geq1$.
            From $\left(\ref{eq:kappa-int-cell}\right)$, observe first that $\kappa_{\intr}(\Delta)\leq0$.
            Indeed, if $\kappa_{\intr}(\Delta)$ were positive, then each term in the sum would have to equal zero, so that we would have
            \[
                i_k<j_k\quad\textrm{and}\quad j_k-1<i_{k+1}
            \]
            for all $k$, implying that $i_1<j_1\leq i_2<j_2\leq\cdots\leq i_r<j_r\leq i_1$, which is impossible.
            
            Further, note that each $k\in\left\{1,\dots, r\right\}$ such that $j_k=1$ adds an extra contribution of $-2\pi$ to the right-hand side of $\left(\ref{eq:kappa-int-cell}\right)$.
            
            Wrapping everything up, the above computations show that transition vertices and vertex discs have non-positive (interior) curvature; moreover, every $a$-boundary arc $\alpha$ with index $1$ will contribute $-2\pi$ to the curvature
            \begin{itemize}
                \item Of a transition vertex if $\alpha$ is preceded by a $b$-boundary arc along $\partial\Sigma$, as in Figure \ref{fig:trans-vrtx} with $j=1$, or
                \item Of a vertex disc if $\alpha$ is preceded by an edge that is part of a vertex disc, as in Figure \ref{fig:cell-disc} with for instance $j_1=1$.
            \end{itemize}
            Since the indices along each component of $\partial\Sigma$ read $1,1,\dots,\ell,\ell$, alternating between $a$- and $b$-boundary arcs, with a total of $n(\Sigma)$ repetitions of this pattern across all boundary components of $\Sigma$, there are in total $n(\Sigma)$ $a$-boundary arcs with index $1$, and their total contribution to the curvature of transition vertices and vertex discs is at most $-2\pi\cdot n(\Sigma)$ as wanted.
        \end{proof}
        
        Using the decomposition of $\Sigma$ into vertex discs and $a$- and $b$-regions to compute its total curvature via Corollary \ref{cor:surf-dec} now yields
        \begin{align*}
            \kappa(\Sigma)&=\sum_{\substack{\textrm{$v$ arc}\\\textrm{endpoint}}}\overbrace{\kappa(v)}^{=0}
            +\overbrace{\sum_{\substack{\textrm{$v$ transition}\\\textrm{vertex}}}\kappa(v)
            +\sum_{\substack{\textrm{$\Delta$ vertex}\\\textrm{disc}}}\kappa_{\intr}(\Delta)}^{\leq-2\pi\cdot n(\Sigma)}
            +\sum_{\substack{\textrm{$\Lambda$ $a$- or}\\\textrm{$b$-region}}}\overbrace{\kappa_{\intr}(\Lambda)}^{\leq0}\\
            &\leq-2\pi\cdot n(\Sigma).\qedhere
        \end{align*}
    \end{proof}
    
\bibliography{PhD-Refs.bib}

\end{document}